\documentclass[12pt]{article}

\usepackage[T1]{fontenc}
\usepackage{lmodern}
\usepackage{amsmath,amssymb,amsfonts,amsthm,mathtools}
\usepackage{enumitem}
\usepackage{xcolor}
\usepackage{graphicx}
\usepackage{tikz}
\usetikzlibrary{calc,shapes.geometric}
\usepackage[labelsep=period]{caption}

\usepackage{needspace}
\usepackage[pdftitle={A shorter proof of the linear list Hadwiger conjecture},
  pdfauthor={Yangyan Gu and Rongxing Xu},
  pdfsubject={A simplified proof of a linear list coloring bound},
  pdfkeywords={list coloring, complete minors, connectivity},
  pdfstartview=XYZ,bookmarks=true,colorlinks=true,
  linkcolor=blue,urlcolor=blue,citecolor=blue,
  linktocpage=true,hyperindex=true]{hyperref}

\allowdisplaybreaks
\numberwithin{equation}{section}
\setlist[enumerate]{label=\textup{(\arabic*)},ref=(\arabic*),leftmargin=2.35em,itemsep=0.2em,topsep=0.3em}
\setlist[itemize]{leftmargin=2em,itemsep=0.2em,topsep=0.3em}

\newtheorem{theorem}{Theorem}[section]
\newtheorem{lemma}[theorem]{Lemma}

\newtheorem{observation}[theorem]{Observation}

\newtheorem{claim}{Claim}[section]
\newtheorem{conjecture}[theorem]{Conjecture}
\theoremstyle{definition}
\newtheorem{definition}[theorem]{Definition}
\newtheorem{remark}[theorem]{Remark}

\newcommand{\cst}[1]{C_{\ref{#1}}}
\newcommand{\tst}[1]{T_{\ref{#1}}}
\newcommand{\ceil}[1]{\left\lceil #1\right\rceil}
\newcommand{\calP}{\mathcal P}
\newcommand{\calT}{\mathcal T}
\newcommand{\calQ}{\mathcal Q}
\newcommand{\calM}{\mathcal M}
\DeclareMathOperator{\ch}{ch}
\DeclareMathOperator{\Prob}{\mathbb P}
\DeclareMathOperator{\Ex}{\mathbb E}

\title{\large{\bfseries A shorter proof of the linear list Hadwiger conjecture}}
\author{Yangyan Gu\thanks{School of Mathematics, Physics and Information Science, Shaoxing University, Shaoxing, Zhejiang, 312000, China.
		Email: \href{mailto:yygu@usx.edu.cn}{\textcolor{black}{\texttt{yygu@usx.edu.cn}}}. Supported by the National Natural Science Foundation of China
		(Young Scientists Fund, Grant No.~12601684;
		General Program, Grant No.~12671412).}
\qquad Rongxing Xu\thanks{School of Mathematical Sciences, Zhejiang Normal University,
Jinhua, Zhejiang, 321004, China.
Email: \href{mailto:xurongxing@zjnu.edu.cn}{\textcolor{black}{\texttt{xurongxing@zjnu.edu.cn}}}.
Supported by the National Natural Science Foundation of China
(Young Scientists Fund, Grant No.~12401472).}
}
\date{\today}

\begin{document}
\maketitle
\vspace{-2em}

\begin{abstract}
The linear list Hadwiger conjecture asserts that every $K_t$-minor-free
graph is $O(t)$-choosable.
Following an approach outlined by Delcourt and Postle, we recently
established $O(t\log\log t)$-choosability.
Soon afterward, OpenAI provided a proof of the conjecture.
A key new ingredient of their proof is a linear bound on the choice number
of $K_t$-minor-free graphs with at most $t^{11/10}$ vertices.
In this paper, we give a relatively simple proof of this small-graph bound for
graphs with at most $t(\ln t)^5$ vertices.
Combining this with the argument from our earlier
$O(t\log\log t)$ bound, we obtain a shorter and somewhat more accessible
proof of the conjecture.
\end{abstract}

\section{Introduction}\label{sec:introduction}

All graphs are finite and simple. A graph $G$ contains an \emph{$H$ minor} if a
graph isomorphic to $H$ can be obtained from a subgraph of $G$ by contracting edges.

Hadwiger's conjecture~\cite{Hadwiger1943} states that every graph with no
$K_t$ minor is $(t-1)$-colorable. Kostochka~\cite{Kostochka1984} and
Thomason~\cite{Thomason1984} independently proved density bounds giving
$O(t\sqrt{\log t})$-colorability. Norin, Postle and
Song~\cite{NorinPostleSong2023} improved this to
$O(t(\log t)^\beta)$ for every fixed $\beta>1/4$.
Delcourt and Postle~\cite{DelcourtPostle2025} then obtained
$O(t\log\log t)$-colorability. Liu and Luo~\cite{LiuLuo2026}
subsequently improved this to $O(t\log\log\log t)$-colorability.
More recently, Norin and Steiner~\cite{NorinSteiner2026} proved the
linear Hadwiger conjecture, showing that every $K_t$-minor-free graph is
$O(t)$-colorable.

A \emph{list assignment} $L$ on a graph $G$ assigns a finite set
$L(v)$ of colors to each $v\in V(G)$.
For an integer $k\ge0$, a list assignment $L$ is a \emph{$k$-list assignment}
if $|L(v)|=k$ for every $v\in V(G)$.
An \emph{$L$-coloring} of $G$ is a function $\phi$ on $V(G)$ such that
$\phi(v)\in L(v)$ for every $v\in V(G)$ and $\phi(u)\ne\phi(v)$ for every $uv\in E(G)$.
A graph $G$ is \emph{$k$-choosable} if $G$ has an $L$-coloring
for every $k$-list assignment $L$ on $G$.
The \emph{choice number} (also called the \emph{list chromatic number}) of $G$,
denoted by $\ch(G)$, is the least nonnegative integer $k$ for which $G$ is
$k$-choosable.

The direct list version of Hadwiger's conjecture is false.
Voigt~\cite{Voigt1993} constructed a planar graph that is not
$4$-choosable. Bar\'at, Joret and Wood~\cite{BaratJoretWood2011}
showed that, for every integer $q\ge1$, there is a graph with no
$K_{3q+2}$ minor that is not $4q$-choosable.
Steiner~\cite{Steiner2022} later proved a lower bound of $(2-o(1))t$
for the maximum choice number of a $K_t$-minor-free graph.
Kawarabayashi and Mohar~\cite{KawarabayashiMohar2007} proposed the
following weaker conjecture.

\begin{conjecture}\label{conj:list-linear}
There is an absolute positive integer $C$ such that every $K_t$-minor-free
graph is $Ct$-choosable for every integer $t\ge2$.
\end{conjecture}

The classical density bound also implies
$O(t\sqrt{\log t})$-choosability. Norin and
Postle~\cite{NorinPostle2023} improved this to
$O(t(\log t)^\beta)$ for each fixed $\beta>1/4$, and
Postle~\cite{Postle2020} obtained $O(t(\log\log t)^6)$-choosability.
At the end of the first version~\cite{DelcourtPostle2021v1}
of~\cite{DelcourtPostle2025}, the authors noted that their methods and
those of~\cite{Postle2020} yield $O(t(\log\log t)^2)$-choosability.
The details were left to the reader.
In our earlier paper~\cite{GuXu2026}, we obtained
$O(t\log\log t)$-choosability by improving the bound used to color small
vertex sets in the construction.

Recently, OpenAI~\cite{OpenAI2026} provided a proof of Conjecture~\ref{conj:list-linear}.
An important ingredient of their proof is a linear bound on the choice
number of $K_t$-minor-free graphs with at most $t^{11/10}$ vertices.
For graphs with at most $t(\ln t)^5$ vertices, we give a relatively simple
proof of this small-graph bound.
Combining this restricted bound with established structural tools and
earlier separation and recursive constructions~\cite{Postle2020,DelcourtPostle2025,GuXu2026},
we obtain a shorter and somewhat more accessible proof of the conjecture.

\begin{theorem}\label{thm:main}
There is an absolute positive integer $C=\cst{thm:main}$ such that, for every integer $t\ge2$,
every $K_t$-minor-free graph is $Ct$-choosable.
\end{theorem}

The proof is by contradiction. Starting with a graph that is not
colorable from large lists, we find disjoint subgraphs that are still
not colorable from large sublists. Complete minors constructed
recursively in these subgraphs are joined into a larger complete minor.
The construction eventually produces a $K_t$ minor.

Section~\ref{sec:deletion} proves the bound for small graphs and develops tools for
vertex deletion and connectivity. Section~\ref{sec:fans} collects the
path and minor tools. Section~\ref{sec:separation} proves the separation
lemma, and Section~\ref{sec:main-proof} completes the recursive construction.

We end this section with some notation.
For a graph $G$ and a vertex set $X$, let $G[X]$ denote the subgraph
of $G$ induced by $V(G)\cap X$, and let
$G\setminus X=G[V(G)\setminus X]$.
The numbers of vertices and edges of $G$ are $\mathsf{v}(G)$ and
$\mathsf{e}(G)$, respectively. For a nonempty graph $G$, its \emph{density} is
$\mathsf{d}(G)=\mathsf{e}(G)/\mathsf{v}(G)$, which is half its average degree.
We write $N_G(v)$ for the set of neighbors of $v$ in $G$ and
$\delta(G)$ for the minimum degree of $G$.
For $X\subseteq V(G)$, its \emph{closed neighborhood} is
$N_G[X]=X\cup\bigcup_{v\in X}N_G(v)$.
The \emph{independence number} $\alpha(G)$ is the maximum order of an
independent set in $G$.
Paths and subgraphs are \emph{disjoint} if they have no common vertex.
For a family $\mathcal H$ of subgraphs, write
$V(\mathcal H)=\bigcup_{H\in\mathcal H}V(H)$.

For a list assignment $L$ on a nonempty graph $G$, write
$|L|=\min_{v\in V(G)}|L(v)|$. A list assignment $L$ is \emph{bad}
on $G$ if $G$ has no $L$-coloring.
For a subgraph $H$ of $G$, an $L$-coloring of $H$ is a proper
coloring $\phi$ of $H$ with $\phi(v)\in L(v)$ for every $v\in V(H)$.
A list assignment $L'$ on $H$ is a \emph{sublist assignment} of $L$,
written $L'\subseteq L$, if $L'(v)\subseteq L(v)$ for every $v\in V(H)$.

For a nonnegative integer $k$, let $[k]=\{1,\ldots,k\}$, with $[0]=\varnothing$.
We use $\ln$ for the natural logarithm.
For constants, numerical subscripts on $C$ and $T$ refer to the results
in which they are introduced. When a constant is introduced as $C=C_i$, we use $C$
within that statement and its proof, and retain the subscript when
referring to the constant elsewhere.
% Intersections involving paths or minor models refer to their vertex sets.
% All constants are absolute unless we state their dependence.
% In Lemma~\ref{thm:recursive}, $N$ is the connectivity constant and
% $\cst{thm:recursive}$ controls the list size.

\section{List coloring and connectivity}\label{sec:deletion}

A linear bound on the choice number of small graphs allows us to delete
vertices while retaining large bad sublists. We then use these sublists
to find highly connected subgraphs.

\subsection{List coloring of small graphs}\label{subsec:small-list-proof}

Recently, OpenAI~\cite{OpenAI2026} (Proposition~3.1) proved that every
$K_t$-minor-free graph with at most $t^{11/10}$ vertices is $O(t)$-choosable.
We need only the following weaker form, whose proof is postponed
to the end of this subsection.

\begin{lemma}\label{thm:local-list}
There are absolute integers $C=\cst{thm:local-list}\ge2$ and $\tst{thm:local-list}\ge3$ such that,
for every integer $t\ge \tst{thm:local-list}$, every $K_t$-minor-free graph
on at most $t(\ln t)^5$ vertices is $Ct$-choosable.
\end{lemma}

A \emph{fractional coloring} of $G$ assigns nonnegative weights $w(I)$
to its independent sets so that $\sum_{I\ni v}w(I)\ge1$ for every vertex
$v$. Its minimum total weight is the \emph{fractional chromatic number}
$\chi_f(G)$. By linear programming duality, $\chi_f(G)$ also equals
the maximum of $\sum_{v\in V(G)}x_v$ over nonnegative vertex weights
$x_v$ satisfying $\sum_{v\in I}x_v\le1$ for every independent set
$I$ of $G$.

Reed and Seymour~\cite{ReedSeymour1998} proved the following.

\begin{theorem}[\cite{ReedSeymour1998}]\label{thm:local-fractional}
For every integer $t\ge2$, every $K_t$-minor-free graph $G$ satisfies
$\chi_f(G)\le2(t-1)$.
\end{theorem}

If $\chi_f(G)\le r$ with $r>0$, dividing a fractional coloring by $r$ and
assigning any missing probability to the empty set gives a distribution
on independent sets containing every vertex with probability at least
$1/r$. Restriction cannot increase $\chi_f$, and adding an independent
set increases it by at most one.

For an integer $k\ge1$, a graph is \emph{$k$-connected} if it has more than $k$ vertices and
deleting fewer than $k$ vertices leaves it connected. Its vertex
connectivity is denoted by $\kappa(G)$.

The following lemma follows from the proof of Gir\~ao and
Narayanan~\cite{GiraoNarayanan2022} and was also used by
Delcourt and Postle~\cite{DelcourtPostle2025} (Theorem~5.1).

\begin{lemma}[\cite{GiraoNarayanan2022}]\label{thm:local-ordinary-extraction}
For every integer $k\ge1$, every graph $G$ with $\chi(G)\ge7k$ contains
a $k$-connected subgraph $J$ with $\chi(J)\ge\chi(G)-6k$.
\end{lemma}

For a graph $G$ and an integer $s\ge1$, the \emph{strong product}
$G\boxtimes K_s$ has vertex set $V(G)\times[s]$.
Two distinct vertices $(u,i)$ and $(v,j)$ are adjacent if and only if
$u=v$ or $uv\in E(G)$. Each $C_v=\{v\}\times[s]$ is therefore a clique
of order $s$. Since the second factor is complete, this is also the
\emph{lexicographic product} of $G$ and $K_s$
(see pp.~51 and~56 in \cite{HammackImrichKlavzar2011}).
The following lemma follows from
\cite[Corollary~3.4.5]{ScheinermanUllman1997} and
\cite[Theorem~26.15]{HammackImrichKlavzar2011}.

\begin{lemma}\label{lem:clique-product}
For every graph $G$ and integer $s\ge1$, we have
$\chi_f(G\boxtimes K_s)=s\chi_f(G)$.
Moreover, for every graph $G$, there is a positive integer $d$ such that
$\chi(G\boxtimes K_d)=d\chi_f(G)$.
\end{lemma}

The following lemma is the fractional version of
Lemma~\ref{thm:local-ordinary-extraction}.

\begin{lemma}\label{lem:fractional-extraction}
For every integer $k\ge1$, a graph $G$ with $\chi_f(G)\ge7k$ has an
induced $k$-connected subgraph $J$ with $\chi_f(J)\ge\chi_f(G)-6k$.
\end{lemma}

\begin{proof}
For each induced subgraph $F$ of $G$, choose a positive integer $d_F$
as in Lemma~\ref{lem:clique-product}, and let $s$ be the least common
multiple of these integers. For each $F$, partition $[s]$ into $s/d_F$
blocks of size $d_F$ and color the corresponding copies of
$F\boxtimes K_{d_F}$ with disjoint palettes, using
$(s/d_F)\cdot d_F\cdot\chi_f(F)=s\chi_f(F)$ colors in total. Also by Lemma~\ref{lem:clique-product},
$\chi(F\boxtimes K_s)\ge\chi_f(F\boxtimes K_s)=s\chi_f(F)$, and hence
$\chi(F\boxtimes K_s)=s\chi_f(F)$ for every induced subgraph $F$ of $G$.

Now $\chi(G\boxtimes K_s)=s\chi_f(G)\ge7sk$.
By Lemma~\ref{thm:local-ordinary-extraction}, $G\boxtimes K_s$ has an
$sk$-connected subgraph $J'$ with
$\chi(J')\ge s\chi_f(G)-6sk$.
Let $J$ be the subgraph of $G$ induced by the vertices $v$ for which
$C_v$ contains a vertex of $J'$. Since $J'$ has more than $sk$ vertices
and each $C_v$ has only $s$ vertices, $J$ has more than $k$ vertices.
For any $X\subseteq V(J)$ with $|X|<k$, delete from $J'$ all copies of
vertices in $X$. Fewer than $sk$ vertices are deleted, so the remaining
graph is connected. Replacing each remaining copy by its original
vertex produces a connected subgraph spanning $J\setminus X$.
Hence $J\setminus X$ is connected, and $J$ is $k$-connected.
Finally, $J'\subseteq J\boxtimes K_s$ and the identity above applied to
$J$ imply
$\chi_f(J)=\chi(J\boxtimes K_s)/s\ge\chi(J')/s\ge\chi_f(G)-6k$.
\end{proof}

A graph is \emph{$k$-knit} if it has at least $k$ vertices and every
partition of $k$ distinct vertices into nonempty parts can be contained
in pairwise disjoint connected subgraphs, one for each part.
Delcourt and Postle~\cite{DelcourtPostle2025} proved the following as the
case $\ell=1$ of Theorem~5.8 in their paper.

\begin{lemma}[\cite{DelcourtPostle2025}]\label{thm:knit}
There is an absolute integer $C=\cst{thm:knit}\ge1$ such that every
$Ck$-connected graph is $k$-knit, for every integer $k\ge1$.
\end{lemma}

The next lemma uses fixed matchings to check connectivity after deletion.

\begin{lemma}\label{lem:local-cut-tests}
Let $G$ be an $n$-vertex $k$-connected graph with $\alpha(G)\le s$,
where $s$ is a positive integer and $k$ is sufficiently large.
There is a family $\mathcal M$ of at most $(n+1)^s$ matchings, each of
order $\lfloor k/8\rfloor$, such that for every $Z\subseteq V(G)$
satisfying the conditions
\begin{enumerate}
\item every vertex has at least $k/2$ neighbors in $V(G)\setminus Z$, and
\item each matching in $\mathcal M$ has fewer than $k/32$ endpoints in $Z$,
\end{enumerate}
we have that $G\setminus Z$ is $\lfloor k/32\rfloor$-connected.
\end{lemma}

\begin{proof}
For each $S\subseteq V(G)$ with $|S|\le s$ such that $N_G[S]$ and its
complement both have order at least $k/3$, fix a matching $M_S$ of order
$\lfloor k/8\rfloor$ between $N_G[S]$ and $V(G)\setminus N_G[S]$.
If no such matching existed, the endpoints of a maximum crossing
matching would separate the two nonempty sides after fewer than $k/4$
vertices were deleted, contrary to $k$-connectivity.
Let $\mathcal M$ be the family of these matchings. Thus we have
\[
 |\mathcal M|\le\sum_{i=0}^{s}\binom{n}{i}
 \le\sum_{i=0}^{s}n^i\le(n+1)^s.
\]

Fix $Z$ satisfying (1) and (2).
By (1), $G\setminus Z$ has minimum degree at least $k/2$, and hence
more than $\lfloor k/32\rfloor$ vertices.
Suppose that deleting $X\subseteq V(G)\setminus Z$ of order less than $k/32$ leaves two
nonempty sides $Y_1,Y_2$ with no edges between them.
Every vertex of either side has more than $k/3$ neighbors in that side.
Choose a maximal independent set $S$ of $G[Y_1]$.
Then $|S|\le s$, $Y_1\subseteq N_G[S]$ by maximality, and
$Y_2\cap N_G[S]=\varnothing$.
Since $|Y_1|,|Y_2|>k/3$, we have $M_S\in\mathcal M$.
Every edge of $M_S$ meets $Z\cup X$, since otherwise it would join
$Y_1$ to $Y_2$.
But (2) and $|X|<k/32$ imply that $Z\cup X$ contains fewer than $k/16$
of its endpoints. This is too few to meet $\lfloor k/8\rfloor$ disjoint
edges, a contradiction.
\end{proof}

A \emph{$K_q$ model} consists of $q$ pairwise disjoint nonempty connected
subgraphs, called \emph{branch sets}, with an edge between every two.
For the random choices in the thinning argument, we use the following
bounded differences inequality of McDiarmid~\cite{McDiarmid1989}.

\begin{lemma}[\cite{McDiarmid1989}]\label{lem:bounded-differences}
Let $X_1,\ldots,X_m$ be independent random variables, and let $f$ be
a real-valued function of $m$ variables.
Suppose that changing any one input of $f$, while keeping all other
inputs fixed, changes its value by at most $c>0$.
Then, for $Y=f(X_1,\ldots,X_m)$ and every $a>0$, we have
\[
 \Prob\bigl(|Y-\Ex[Y]|\ge a\bigr)
 \le2\exp\left(-\frac{2a^2}{mc^2}\right).
\]
\end{lemma}

The next lemma is the form of the thinning argument of
OpenAI~\cite{OpenAI2026} that we use at successive stages.

\Needspace{9\baselineskip}
\begin{lemma}\label{lem:eligible-thinning}
Let $\beta\ge32$ be a fixed absolute integer satisfying
\mbox{$2^{15}\cst{thm:knit}e^{-\beta/4}\le1$}.
After increasing $\tst{thm:local-list}$ if necessary, the following holds.
Let $t\ge \tst{thm:local-list}$ and $s\ge0$ be integers with
$r=t2^{1-s}\ge t/\ln t$ and $N=\ceil{\beta(s+1)r}$.
Suppose that $G$ is a $K_t$-minor-free graph with
$\mathsf v(G)\le n=t(\ln t)^5$ and $\chi_f(G)\le r$.
If $A_1,\ldots,A_N$ are fixed subsets of $V(G)$ such that every vertex
of $G$ belongs to at least $N/2$ of these sets, then there exist
independent sets $I_1,\ldots,I_N$ of $G$ such that
\[
 \chi_f\left(G\mathbin{\bigg\backslash}\bigcup_{i=1}^N(I_i\cap A_i)\right)\le \frac{r}{2}.
\]
\end{lemma}

\begin{proof}
We repeatedly choose an independent set $I$ of $G$ of order $\ceil{n/r}$
consisting of undeleted vertices, together with an unused index $i$
satisfying $|I\cap A_i|\ge |I|/4$.
At each step, we record $I_i=I$ as one of the independent sets required
by the lemma and delete the vertices of $I\cap A_i$.
The process continues until no further deletion is possible.
Let $H$ be the induced subgraph of $G$ remaining at this point, and let
$K\subseteq[N]$ be the set of unused indices.
Each deletion removes at least $n/(4r)$ vertices, so the process uses
at most $4r\le N/8$ indices in total.
As every vertex of $G$ belongs to at least $N/2$ of the sets $A_i$,
at least $N/2-N/8=3N/8>N/4$ of these sets have indices in $K$.
Suppose that $H$ contains an independent set $I$ of order $\ceil{n/r}$.
Then $\sum_{i\in K}|I\cap A_i|\ge N|I|/4$.
Since $|K|\le N$, some $i\in K$ satisfies $|I\cap A_i|\ge |I|/4$.
This would allow another deletion, contrary to the termination of the
process. Hence $\alpha(H)\le\ceil{n/r}-1<n/r$.

Let $k=32\cst{thm:knit}t(2^{s+4}-1)$.

\begin{claim}\label{clm:thinning-families}
There are a family $\mathcal H$ of induced subgraphs of $H$ and a family
$\mathcal C$ of vertex subsets of $H$ such that
$H=\bigcup_{F\in\mathcal H}F$, $|\mathcal H|\le n+1$, and $|\mathcal C|\le n$.
Every $F\in\mathcal H$ has at most $2k$ vertices or is $k$-connected,
and $|C|<k$ for every $C\in\mathcal C$.
Moreover, every subgraph $J$ of $H$ such that $J\setminus C$ is
connected for every $C\in\mathcal C$ is contained in some $F\in\mathcal H$.
\end{claim}

\begin{proof}
Let $\mathcal H_0=\{H\}$ and $\mathcal C_0=\varnothing$.
For $i=1,2,\ldots$, if $\mathcal H_{i-1}$ contains a subgraph of
order greater than $2k$ that is not $k$-connected, choose one such
subgraph $F$.
For this subgraph $F$, let $C$ be a vertex separator with $|C|<k$,
and let $Y_1,Y_2$ be nonempty sets partitioning $V(F)\setminus C$,
with no edges between them.
Let $F_1=F\setminus Y_2$ and $F_2=F\setminus Y_1$.
The family $\mathcal H_i$ is obtained from $\mathcal H_{i-1}$ by
replacing $F$ with $F_1,F_2$, and $\mathcal C_i$ is obtained from
$\mathcal C_{i-1}$ by adding $C$.
The process continues until no such subgraph $F$ exists.

We now show that at most $n$ replacements can occur.
For $i\ge0$, let
\[
 \tau(\mathcal H_i)=\sum_{J\in\mathcal H_i}\max\{\mathsf v(J)-k,0\}.
\]
Then $\tau(\mathcal H_0)\le n$.
Suppose that $\mathcal H_i$ is obtained from $\mathcal H_{i-1}$ by
replacing $F$ with $F_1,F_2$, labeled so that
$\mathsf v(F_1)\le\mathsf v(F_2)$.
Since $\mathsf v(F_1)+\mathsf v(F_2)=\mathsf v(F)+|C|$, we have
$\mathsf v(F_2)\ge\mathsf v(F)/2>k$.
It follows that
\[
 \tau(\mathcal H_{i-1})-\tau(\mathcal H_i)
 \ge\mathsf v(F_1)-|C|-\max\{\mathsf v(F_1)-k,0\}
 =\min\{\mathsf v(F_1),k\}-|C|\ge1.
\]
The last inequality uses $\mathsf v(F_1)>|C|$ and $k>|C|$.
Since $\tau(\mathcal H_i)$ is nonnegative, the process terminates
after at most $n$ replacements.

Suppose that the process terminates after $\ell$ replacements, and let
$\mathcal H=\mathcal H_\ell$ and $\mathcal C=\mathcal C_\ell$.
Each replacement records at most one separator and increases the number
of subgraphs by at most one, so $|\mathcal C|\le n$ and
$|\mathcal H|\le n+1$.
Each replacement preserves the union of the subgraphs, so
$H=\bigcup_{F\in\mathcal H}F$.
By the stopping condition, every $F\in \mathcal H$ has at most $2k$
vertices or is $k$-connected.

Suppose that $J$ is a subgraph of $H$ such that $J\setminus C$ is
connected for every $C\in\mathcal C$. To show that there exists some graph in $\mathcal H$ containing $J$,
it suffices to show that some member of $\mathcal H_i$ contains $J$
for every $i=0,\ldots,\ell$. We prove this by induction on $i$.
The case $i=0$ follows from $\mathcal H_0=\{H\}$.
Suppose that $F\in\mathcal H_{i-1}$ contains $J$.
If $F\in\mathcal H_i$, we are done.
Otherwise, $\mathcal H_i$ contains $F_1=F\setminus Y_2$ and
$F_2=F\setminus Y_1$, where $C,Y_1,Y_2$ are the sets chosen at step $i$.
The vertices of $J\setminus C$ lie entirely in one of $Y_1,Y_2$,
since $J\setminus C$ is connected and there are no edges between these sets.
Since $V(F_1)\cap V(F_2)=C$, one of $F_1,F_2$ contains $J$, as desired.
\end{proof}

Let $\mathcal H,\mathcal C$ be the families from
Claim~\ref{clm:thinning-families}.
Since $\chi_f(H)\le r$, let $w(I)$ be fractional coloring weights
on $H$ of total weight $r$, adding any missing weight to the empty set.
For each $i\in K$, independently choose an independent set $I_i$ of $H$
with $\Prob(I_i=I)=w(I)/r$.
Let $R=V(H)\setminus\bigcup_{i\in K}(I_i\cap A_i)$.
For every $i\in K$ and $v\in V(H)\cap A_i$,
$\Prob(v\in I_i)=\sum_{I\ni v}w(I)/r\ge1/r$.
As the greedy deletion process used at most $N/8$ indices,
each vertex of $H$ still belongs to at least $N/2-N/8=3N/8>N/4$
of the sets $A_i$ with indices in $K$.
By the definition of $R$, $v\in R$ if and only if $v\notin I_i$
for every $i\in K$ with $v\in A_i$.
There are more than $N/4$ such indices, and for each of them
$\Prob(v\notin I_i)\le1-1/r$.
Thus,
\begin{equation}\label{eq:uncovered-probability}
 \Prob(v\in R)
 =\prod_{i\in K,\,A_i\ni v}\Prob(v\notin I_i)
 \le\left(1-\frac1r\right)^{N/4}\le e^{-\frac{N}{4r}}.
\end{equation}

For $i\in[N]\setminus K$, let $I_i$ be the independent set recorded
for index $i$ during the greedy deletions.
Then $G\setminus\bigcup_{i=1}^N(I_i\cap A_i)=H[R]$.
If the independent sets $\{I_i\}_{i\in K}$ can be chosen so that
$\chi_f(H[R])\le r/2$, then the lemma holds.
We may therefore assume that $\chi_f(H[R])>r/2$ for every choice of
$\{I_i\}_{i\in K}$.

\begin{claim}\label{clm:thinning-core}
There is a fixed $k$-connected induced subgraph $J$ of $H$ such that
$\Ex[\chi_f(J[R])]>r/4$.
\end{claim}
\begin{proof}
Let $\mathcal X=\mathcal X_1\cup\mathcal X_2\cup\mathcal X_3$, where
the three families of random variables are defined by
\[
 \mathcal X_1=\{|C\cap R|\}_{C\in\mathcal C},\quad
 \mathcal X_2=\{|V(F)\cap R|\}_{F\in\mathcal H,\,\mathsf v(F)\le2k},\quad
 \mathcal X_3=\{\chi_f(F[R])\}_{F\in\mathcal H,\,\mathsf v(F)>2k}.
\]
Since $N\ge\beta(s+1)r$ and $4e^{-\beta/4}\le4e^{-8}<1$,
by the choice of $\beta$ we have
\begin{equation}\label{eq:uncovered-expectation}
 2ke^{-\frac{N}{4r}}
 <2^{2s+9}\cst{thm:knit}r e^{-\frac{\beta(s+1)}4}
 =2^9\cst{thm:knit}r e^{-\frac\beta4}
   \bigl(4e^{-\frac\beta4}\bigr)^s
 \le2^9\cst{thm:knit}r e^{-\frac\beta4}\le r/64.
\end{equation}
For every $X\in\mathcal X_1\cup\mathcal X_2$,
by Inequalities~\eqref{eq:uncovered-probability} and
\eqref{eq:uncovered-expectation} and linearity of expectation,
$\Ex[X]\le2ke^{-N/(4r)}\le r/64$.

Every $X\in\mathcal X$ depends on the same $|K|\le N$ independent
choices and changes by at most $2n/r$ when one $I_i$ is replaced.
Indeed, such a replacement may add a subset of the original independent
set to $R$ and remove a subset of the new independent set from $R$.
Both subsets are independent and have order less than $n/r$,
so every $X\in\mathcal X_1\cup\mathcal X_2$ changes by at most $2n/r$.
Deleting vertices cannot increase $\chi_f$, while adding an independent
set increases $\chi_f$ by at most one.
Comparing in both directions, every $X\in\mathcal X_3$ changes by at
most one, which is at most $2n/r$ for sufficiently large $t$.

Since $s+1\le2^s$ and both $\beta$ and $t$ are integers, we have
$N\le2\beta t$.
By Lemma~\ref{lem:bounded-differences}, using
$r\ge t/\ln t$ and $n=t(\ln t)^5$, for every $X\in\mathcal X$ we have
\[
 \Prob\left(|X-\Ex[X]|\ge\frac r{64}\right)
 \le2\exp\left(-\frac{r^4}{2^{13}Nn^2}\right)
 \le2\exp\left(-\frac{t}{2^{14}\beta(\ln t)^{14}}\right)
 \le2e^{-\sqrt t}.
\]
By Claim~\ref{clm:thinning-families}, $|\mathcal X|\le2n+1$.
By the union bound, we have
\[
 \Prob\left(\max_{X\in\mathcal X}|X-\Ex[X]|\ge\frac r{64}\right)
 \le2(2n+1)e^{-\sqrt t}<1.
\]
Thus the independent sets $\{I_i\}_{i\in K}$ can be chosen so that
$|X-\Ex[X]|<r/64$ for every $X\in\mathcal X$.
Together with Inequality~\eqref{eq:uncovered-expectation}, it follows
that this choice satisfies
\begin{equation}\label{eq:core-mean-comparison}
 \begin{gathered}
  |C\cap R|<\frac r{32}\quad(C\in\mathcal C),\qquad
  |V(F)\cap R|<\frac r{32}\quad(F\in\mathcal H,\ \mathsf v(F)\le2k),\\
  \Ex[\chi_f(F[R])]>\chi_f(F[R])-\frac r{64}
  \quad(F\in\mathcal H,\ \mathsf v(F)>2k).
 \end{gathered}
\end{equation}

By our assumption, $\chi_f(H[R])>r/2$.
For large $t$, $r>128$, so
$7\ceil{r/32}<7r/16<r/2$.
By Lemma~\ref{lem:fractional-extraction}, $H[R]$ contains a
$\ceil{r/32}$-connected subgraph $F$ with
\[
 \chi_f(F)\ge\chi_f(H[R])-6\ceil{r/32}
 >\frac r2-6\left(\frac r{32}+1\right)
 =\frac{5r}{16}-6>\frac r4+\frac r{64}.
\]
Since $V(F)\subseteq R$, by Inequalities~\eqref{eq:core-mean-comparison}
we have $|C\cap V(F)|\le|C\cap R|<r/32$ for every $C\in\mathcal C$.
The graph $F$ is $\ceil{r/32}$-connected, so $F\setminus C$ is
connected for every $C\in\mathcal C$.
By Claim~\ref{clm:thinning-families}, some $J\in\mathcal H$ contains $F$.
If $\mathsf v(J)\le2k$, then by Inequalities~\eqref{eq:core-mean-comparison},
$\chi_f(F)\le\mathsf v(F)\le|V(J)\cap R|<r/32$, a contradiction.
Thus $\mathsf v(J)>2k$, and $J$ is $k$-connected by
Claim~\ref{clm:thinning-families}.
Moreover, by Inequalities~\eqref{eq:core-mean-comparison}, we have
$\Ex[\chi_f(J[R])]>\chi_f(J[R])-\frac r{64}
\ge\chi_f(F)-\frac r{64}>\frac r4$.
\end{proof}

Let $J$ be the fixed graph from Claim~\ref{clm:thinning-core}, and let
$p=\ceil{r/32}$ and $q=\binom{2^{s+4}}2$.
Recall that $k=32\cst{thm:knit}t(2^{s+4}-1)$.
By the definitions of $r$ and $q$, we have $k=2\cst{thm:knit}rq$.

\begin{claim}\label{clm:thinning-models}
There is a set $X\subseteq V(J)$ such that $J[X]$ contains $q$
pairwise disjoint $K_{2p}$ models, every vertex of $J$ has at least
$k/2$ neighbors in $V(J)\setminus X$, and $J\setminus X$ is
$k/32$-connected.
\end{claim}
\begin{proof}
For each $i\in K$ and $j\in[q]$, independently choose an independent
set $I_i^j$ of $H$, with $\Prob(I_i^j=I)=w(I)/r$
for every independent set $I$ of $H$.
All these choices are independent of the sets $\{I_i\}_{i\in K}$.
For $j\in[q]$, let
\[
 R_j=V(J)\setminus\bigcup_{i\in K}(I_i^j\cap A_i),\qquad
 S_j=R_j\setminus\bigcup_{h\ne j}R_h,\qquad
 \text{and}\quad X=\bigcup_{j=1}^qR_j.
\]
The sets $S_1,\ldots,S_q$ are pairwise disjoint and contained in $X$.
Each $R_j$ has the same distribution as $R\cap V(J)$, so by
Claim~\ref{clm:thinning-core}, $\Ex[\chi_f(J[R_j])]>r/4$.

By Inequalities~\eqref{eq:uncovered-probability} and
\eqref{eq:uncovered-expectation} and the union bound,
using $k=2\cst{thm:knit}rq$ and $\cst{thm:knit}\ge1$,
for every $v\in V(J)$ we have
\begin{equation}\label{eq:uncovered-union-probability}
 \Prob(v\in X)
 \le\sum_{h=1}^q\Prob(v\in R_h)
 \le qe^{-\frac{N}{4r}}
 \le\frac{qr}{2^7k}=\frac1{2^8\cst{thm:knit}}\le\frac1{64}.
\end{equation}
Since $R_h$ is independent of $R_j$ for $h\ne j$,
by the definition of $S_j$ and Inequality~\eqref{eq:uncovered-union-probability},
for every $j\in[q]$ and $v\in R_j$ we have
\[
 \Prob(v\in S_j\mid R_j)
 =1-\Prob\left(v\in\bigcup_{h\ne j}R_h\mid R_j\right)
 \ge1-\sum_{h\ne j}\Prob(v\in R_h)
 \ge1-qe^{-\frac{N}{4r}}\ge\frac{63}{64}.
\]
For each $j\in[q]$, conditional on $R_j$, each vertex $v\in R_j$
has a nonnegative weight $x_v$ such that
$\sum_{v\in R_j}x_v=\chi_f(J[R_j])$ and
$\sum_{v\in I}x_v\le1$ for every independent set $I$ of $J[R_j]$.
Such weights exist by the dual characterization of $\chi_f$.
Since every independent set of $J[S_j]$ is also independent in $J[R_j]$,
by the dual characterization we have
$\chi_f(J[S_j])\ge\sum_{v\in S_j}x_v$.
Thus, by linearity of expectation, we have
\[
 \Ex[\chi_f(J[S_j])\mid R_j]
 \ge\sum_{v\in R_j}x_v\Prob(v\in S_j\mid R_j)
 \ge\frac{63}{64}\chi_f(J[R_j]).
\]
Since $R_j$ has the same distribution as $R\cap V(J)$ and
$\Ex[\chi_f(J[R])]\le r$, by the law of total expectation we have
\begin{equation}\label{eq:disjoint-mean}
 \Ex[\chi_f(J[S_j])]
 \ge\frac{63}{64}\Ex[\chi_f(J[R])]
 \ge\Ex[\chi_f(J[R])]-\frac r{64}>\frac r4-\frac r{64}
 =\frac{15r}{64}.
\end{equation}

Since $\alpha(J)<n/r\le(\ln t)^6$ and $\mathsf v(J)\le n = t(\ln t)^5$,
there is a fixed family $\mathcal M$ of matchings in $J$ as in
Lemma~\ref{lem:local-cut-tests}, each of order $\lfloor k/8\rfloor$, such that
\[
 |\mathcal M|\le(\mathsf v(J)+1)^{\alpha(J)}
 \le\exp\bigl((\ln t)^6\ln(n+1)\bigr)
 \le\exp(4(\ln t)^7).
\]
The last inequality holds for sufficiently large $t$, since
 $\ln(n+1)\le\ln2+\ln t+5\ln\ln t\le4\ln t$.

Let $\mathcal Y=\mathcal Y_1\cup\mathcal Y_2\cup\mathcal Y_3$, where
the three families of random variables are defined by
\[
 \mathcal Y_1=\{|N_J(v)\cap X|\}_{v\in V(J)},\quad
 \mathcal Y_2=\{|V(M)\cap X|\}_{M\in\mathcal M},\quad
 \mathcal Y_3=\{\chi_f(J[S_j])\}_{j\in[q]}.
\]
Every $Y\in\mathcal Y$ depends on the same $q|K|\le qN$ independent
choices. As in the proof of Claim~\ref{clm:thinning-core}, replacing
one $I_i^h$ changes every $Y\in\mathcal Y$ by at most $2n/r$,
since the vertex sets added to or removed from $X$ and each $S_j$
are independent and have order less than $n/r$.
Since $2^{s+4}=32t/r\le32\ln t$, we have $q<512(\ln t)^2$.
Together with $N\le2\beta t$, this gives $qN\le t(\ln t)^3$
for sufficiently large $t$, uniformly over the allowed values of $s$.
By Lemma~\ref{lem:bounded-differences}, for every $Y\in\mathcal Y$ we have
\begin{equation}\label{eq:local-concentration}
 \Prob\left(|Y-\Ex[Y]|\ge\frac r{64}\right)
 \le2\exp\left(-\frac{r^4}{2^{13}qNn^2}\right)
 \le2\exp\left(-\frac{t}{2^{13}(\ln t)^{17}}\right)
 \le2e^{-\sqrt t}.
\end{equation}
Since $|\mathcal Y|\le n+\exp(4(\ln t)^7)+q$, by
Inequality~\eqref{eq:local-concentration} and the union bound, we have
\[
 \Prob\left(\max_{Y\in\mathcal Y}|Y-\Ex[Y]|\ge\frac r{64}\right)
 \le2\bigl(n+\exp(4(\ln t)^7)+q\bigr)e^{-\sqrt t}<1.
\]
\Needspace{5\baselineskip}
Thus the independent sets $\{I_i^j\}_{i\in K,\,j\in[q]}$ can be chosen
so that $|Y-\Ex[Y]|<r/64$ for every $Y\in\mathcal Y$.
For this choice, we have the following bounds.
\begin{itemize}
 \item Since $J$ is $k$-connected, $|N_J(v)|\ge k\ge r$ for every
 $v\in V(J)$. By Inequality~\eqref{eq:uncovered-union-probability} and
 linearity of expectation, $\Ex[|N_J(v)\cap X|]\le|N_J(v)|/64$,
 so $|N_J(v)\cap X|<\frac{|N_J(v)|+r}{64}\le\frac{|N_J(v)|}{32}$.
 \item For every $M\in\mathcal M$, we have
 $|V(M)|=2\lfloor k/8\rfloor\le k/4$.
 Similarly, $\Ex[|V(M)\cap X|]=\sum_{v\in V(M)}\Prob(v\in X)
 \le\frac{|V(M)|}{64}\le\frac{k}{256}\le\frac{k}{64}$,
 so $|V(M)\cap X|<\frac{k+r}{64}\le\frac k{32}$.
 \item By Inequality~\eqref{eq:disjoint-mean}, for every $j\in[q]$ we have
 $\chi_f(J[S_j])>\frac{15r}{64}-\frac r{64}=\frac{7r}{32}>4p$ for large $t$.
\end{itemize}

By the first two bounds and Lemma~\ref{lem:local-cut-tests},
with $X$ playing the role of $Z$, $J\setminus X$ is $k/32$-connected.
Indeed, $|N_J(v)\setminus X|>\frac{31}{32}|N_J(v)|\ge\frac{31k}{32}>\frac k2$
for every $v\in V(J)$.
By the third bound and Theorem~\ref{thm:local-fractional},
each $J[S_j]$ contains a $K_{2p}$ model.
Since $S_1,\ldots,S_q$ are pairwise disjoint subsets of $X$,
these $q$ models are pairwise disjoint and lie in $J[X]$.
\end{proof}

Let $X$ and the $q$ models be as in Claim~\ref{clm:thinning-models}.
Since $2^{s+4}=32t/r\le32\ln t<t$ and $p=\ceil{t/2^{s+4}}$, we
partition $[t]$ into nonempty blocks $T_1,\ldots,T_{2^{s+4}}$,
each of order at most $p$.
There are exactly $q=\binom{2^{s+4}}2$ unordered pairs of distinct blocks,
and we associate the $q$ models bijectively with these pairs.
For each pair $\{T_i,T_j\}$ with $i<j$ and each $\ell\in T_i\cup T_j$,
we select a branch set $B_\ell^{(i,j)}$ of the associated model,
with the $|T_i\cup T_j|\le2p$ selected branch sets all distinct.
There are $t(2^{s+4}-1)=rq/16$ selected branch sets in total.

Since every vertex of $J$ has at least $k/2>rq/16$ neighbors in
$J\setminus X$, we can choose distinct vertices
$y_\ell^{(i,j)}\in V(J)\setminus X$, each with a neighbor in $B_\ell^{(i,j)}$.
By Claim~\ref{clm:thinning-models}, $J\setminus X$ is $k/32$-connected.
Since $k/32=\cst{thm:knit}rq/16$, by Lemma~\ref{thm:knit},
$J\setminus X$ is $(rq/16)$-knit and contains pairwise disjoint
connected subgraphs $D_1,\ldots,D_t$ such that $D_\ell$ contains
$y_\ell^{(i,j)}$ for every pair $\{T_i,T_j\}$ with $\ell\in T_i\cup T_j$.

For $\ell\in[t]$, let
\[
 F_\ell=J\left[V(D_\ell)\cup\bigcup_{i<j,\,\ell\in T_i\cup T_j}V(B_\ell^{(i,j)})\right].
\]
Every $B_\ell^{(i,j)}$ has an edge to $y_\ell^{(i,j)}\in V(D_\ell)$,
so $F_\ell$ is connected.
The subgraphs $F_1,\ldots,F_t$ are pairwise disjoint, since the
$B_\ell^{(i,j)}$ are pairwise disjoint within $X$ and the $D_\ell$
are pairwise disjoint outside $X$.
For any distinct labels $\ell,\ell'\in[t]$, there are $i<j$ such that
$\ell,\ell'\in T_i\cup T_j$.
The branch sets $B_\ell^{(i,j)}$ and $B_{\ell'}^{(i,j)}$ are adjacent
in the model associated with $\{T_i,T_j\}$, so $F_\ell$ and
$F_{\ell'}$ are adjacent.
Thus $F_1,\ldots,F_t$ form a $K_t$ model in $G$, a contradiction.

Hence the independent sets $\{I_i\}_{i\in K}$ can be chosen so that
$\chi_f(H[R])\le r/2$.
Together with the independent sets recorded during the greedy deletions,
they satisfy the required bound.
\end{proof}

\Needspace{5\baselineskip}
\begin{proof}[Proof of Lemma~\ref{thm:local-list}]
Let $\beta$ be as in Lemma~\ref{lem:eligible-thinning}.
Choose an integer $C>2(8\beta+6)(1+\ln2)$.
For sufficiently large $t$, suppose that $G$ is $K_t$-minor-free
with $\mathsf v(G)\le t(\ln t)^5$ and $L$ is a
$Ct$-list assignment on $G$.

Let $r_j=t2^{1-j}$ for $j\ge0$, and $\ell$ be the least integer
such that $r_\ell<t/\ln t$.
Let $N_j=\ceil{\beta(j+1)r_j}$ for $j<\ell$ and $N_\ell=4t$.

We randomly partition the set $\bigcup_{v\in V(G)}L(v)$ of all
colors appearing in the lists into pairwise disjoint sets $C_i^j$,
where $j\in\{0,1,\ldots,\ell\}$ and $i\in[N_j]$.
For the total number of sets $C_i^j$, we have
\[
 \sum_{j=0}^{\ell}N_j
 \le2\beta t\sum_{j\ge0}(j+1)2^{-j}+\ell+4t
 \le(8\beta+6)t,
\]
since $\ell\le2\ln t\le2t$.
The random partition is obtained by assigning each color
independently and uniformly to one of the sets $C_i^j$.
For fixed $v\in V(G)$ and $j\in\{0,\ldots,\ell\}$, if $L(v)$
intersects fewer than $N_j/2$ of the sets $C_i^j$, then it is
disjoint from $C_i^j$ for some $\ceil{N_j/2}$ indices $i$.
There are $\binom{N_j}{\ceil{N_j/2}}\le2^{N_j}$ choices of these
indices, since $[N_j]$ has $2^{N_j}$ subsets.
Since the colors in $L(v)$ are assigned independently and
$|L(v)|=Ct$, by the union bound over these choices,
the probability that $L(v)$ intersects fewer than $N_j/2$ sets is at most
\[
 2^{N_j}\left(1-\frac{\ceil{N_j/2}}{\sum_{h=0}^{\ell}N_h}\right)^{Ct}
 \le\exp\left(N_j\ln2-\frac{CN_j}{2(8\beta+6)}\right)
 \le e^{-N_j}.
\]
Here we use $1-x\le e^{-x}$ and the bounds on the number of sets and $C$.

For $j<\ell$, the minimality of $\ell$ implies $r_j\ge t/\ln t$.
Since $\beta\ge32$, we have
$N_j=\ceil{\beta(j+1)r_j}\ge r_j\ge t/\ln t$.
Also, $N_\ell=4t\ge t/\ln t$ for sufficiently large $t$.
Thus $N_j\ge t/\ln t$ for every $j=0,\ldots,\ell$.
By the union bound over all $v\in V(G)$ and $j\in\{0,\ldots,\ell\}$,
the probability that any list fails the required intersection
condition is at most
$(\ell+1)|V(G)|e^{-t/\ln t}\le(\ell+1)t(\ln t)^5e^{-t/\ln t}<1$
for sufficiently large $t$.
Thus the sets $C_i^j$ can be chosen so that every list intersects
at least $N_j/2$ of $C_1^j,\ldots,C_{N_j}^j$ for every $j$.

Let $H_0=G$. By Theorem~\ref{thm:local-fractional},
$\chi_f(H_0)<r_0$.
For $j=0,\ldots,\ell-1$, suppose that $\chi_f(H_j)\le r_j$.
For each $i\in[N_j]$, let
$A_i=\{v\in V(H_j):L(v)\cap C_i^j\ne\varnothing\}$.
Every vertex belongs to at least $N_j/2$ of these sets.
Since $r_j\ge t/\ln t$, by Lemma~\ref{lem:eligible-thinning}, there
are independent sets $I_1^j,\ldots,I_{N_j}^j$ of $H_j$ such that
\[
 H_{j+1}=H_j\setminus\bigcup_{i=1}^{N_j}(I_i^j\cap A_i),
 \qquad
 \chi_f(H_{j+1})\le\frac{r_j}{2}=r_{j+1}.
\]
Each deleted vertex $v$ is colored with a color in $L(v)\cap C_i^j$
for an index $i$ with $v\in I_i^j\cap A_i$.
Consequently, the graph $H_\ell$ of uncolored vertices satisfies
$\chi_f(H_\ell)\le r_\ell<t/\ln t$.

There is therefore a distribution on the independent sets of $H_\ell$
in which each vertex is included with probability at least $\ln t/t$.
For each $i\in[4t]$, independently choose an independent set $I_i^\ell$
from this distribution.
A vertex $v\in V(H_\ell)$ is \emph{covered} if $v\in I_i^\ell$
for some index $i$ with $L(v)\cap C_i^\ell\ne\varnothing$.
Each list intersects at least $2t$ of $C_1^\ell,\ldots,C_{4t}^\ell$,
so for every $v\in V(H_\ell)$ we have
$\Prob(v\text{ is uncovered})\le(1-\ln t/t)^{2t}\le t^{-2}$.
By the union bound, the probability that some vertex is uncovered
is at most $\mathsf v(H_\ell)t^{-2}\le(\ln t)^5/t<1$.
Thus the independent sets $I_i^\ell$ can be chosen so that every
vertex of $H_\ell$ is covered.
We color each vertex $v\in V(H_\ell)$ with a color in
$L(v)\cap C_i^\ell$ for an index $i$ with $v\in I_i^\ell$
and $L(v)\cap C_i^\ell\ne\varnothing$.

The sets $C_i^j$ are pairwise disjoint, and all vertices colored
from a fixed $C_i^j$ lie in the independent set $I_i^j$.
Thus the resulting coloring is a proper $L$-coloring of $G$.
Choose $\tst{thm:local-list}\ge3$ so that Lemma~\ref{lem:eligible-thinning} and
all preceding estimates hold for every $t\ge \tst{thm:local-list}$.
\end{proof}

\subsection{Bad list assignments and connectivity}\label{subsec:bad-lists-connectivity}

If $L$ is bad on $G$, every sublist of $L$ on $G$ is also bad.
A bad list assignment on a subgraph remains bad on the induced subgraph
with the same vertex set.
If $X\subseteq V(G)$ has $|X|\le|L|$, color $G[X]$ with distinct colors
from its lists. Deleting $X$ and removing the colors used on neighbors
in $X$ leaves a bad sublist on $G\setminus X$ of size at least $|L|-|X|$;
otherwise a coloring would extend to $G$.

\begin{observation}\label{obs:bad-density}
If $L$ is a bad list assignment on $G$, then $G$ has a nonempty
induced subgraph $H$ that is not $L$-colorable and satisfies
$\delta(H)\ge|L|$ and $\mathsf{d}(H)\ge|L|/2$.
\end{observation}

\begin{proof}
Choose an induced subgraph $H$ vertex-minimal subject to not being $L$-colorable.
For each $v\in V(H)$, an $L$-coloring of $H\setminus\{v\}$ would extend
to $v$ if $d_H(v)<|L(v)|$. Thus $\delta(H)\ge|L|$, and
$\mathsf d(H)\ge\delta(H)/2\ge|L|/2$.
\end{proof}

The following consequence of the unbalanced bipartite density bound
of Norin and Postle~\cite{NorinPostle2023} (Theorem~3.2 in their paper)
controls vertices with many neighbors in a fixed set.

\begin{lemma}[\cite{NorinPostle2023}]\label{lem:bipartite}
There is an absolute constant $C=\cst{lem:bipartite}\ge1$ such that,
for every integer $t\ge3$, if a $K_t$-minor-free bipartite graph has
parts $A,B$ and every vertex of $B$ has at least $2t$ neighbors in $A$,
then $|B|\le C(\ln t)|A|$.
\end{lemma}

For an integer $d\ge0$, a graph is \emph{$d$-degenerate} if every
nonempty subgraph has a vertex of degree at most $d$.
Such a graph is $(d+1)$-choosable by greedy coloring.

The next lemma follows the proof of Lemma~4.11 in~\cite{Postle2020},
using Lemma~\ref{thm:local-list} to color the small set.
Here $\tst{thm:local-list}$ is the threshold from Lemma~\ref{thm:local-list},
which may be increased when further conditions are needed.

\begin{lemma}\label{lem:deletion}
There is an absolute constant $C=\cst{lem:deletion}$ with the following property.
Let $t\ge \tst{thm:local-list}$ and $d\ge0$ be integers, and let $G$ be a $K_t$-minor-free graph.
Let $X,Y\subseteq V(G)$ be disjoint with $|X|\le t(\ln t)^5$.
Suppose that $G[Y]$ is $d$-degenerate and every vertex outside
$X\cup Y$ has at most $d$ neighbors in $Y$.
If $L$ is a bad list assignment on $G$ with
$|L|\ge C(t+d+1)$, then $G\setminus(X\cup Y)$
has a bad sublist $L'\subseteq L$ satisfying
$|L'|\ge|L|-C(t+d+1)$.
\end{lemma}

\begin{proof}
Set $C=3\cst{thm:local-list}+1$.
By deleting extra colors, assume that every list has size $|L|$.
Let $Z=\{v\notin X:|N_G(v)\cap X|\ge2t\}$.
By Lemma~\ref{lem:bipartite}, applied to the bipartite subgraph with
parts $X,Z$ and all edges between them,
$|Z|\le\cst{lem:bipartite}|X|\ln t$.

Let $q=\cst{thm:local-list}t$. Since $|L|\ge C(t+d+1)\ge3q+d+1$,
we can choose a random set $A\subseteq\bigcup_{v\in X\cup Z}L(v)$
by including each color independently with probability $2q/|L|$.
For each $v\in X\cup Z$, the random variable $|L(v)\cap A|$ has mean $2q$.
By the Chernoff bounds~\cite{MitzenmacherUpfal2017} (Theorems~4.4(2) and~4.5(2))
and a union bound, the probability that some such vertex has
$|L(v)\cap A|<q$ or $|L(v)\cap A|>3q$ is at most
\[
 2(|X|+|Z|)e^{-\frac{q}{6}}
 \le2(1+\cst{lem:bipartite}\ln t)t(\ln t)^5e^{-\frac{\cst{thm:local-list}t}{6}}=o(1).
\]
Fix an outcome with $q\le|L(v)\cap A|\le3q$ for every $v\in X\cup Z$.
By Lemma~\ref{thm:local-list}, color $G[X]$ from the lists $L(v)\cap A$.

Delete from each remaining list the colors used on neighbors in $X$.
At most $3q$ colors are deleted at each vertex of $Z$, and fewer than
$2t\le q$ at every other vertex outside $X$.
The remaining lists have size at least $|L|-3q\ge d+1$, so $G[Y]$ has
a coloring from these lists.
Delete the colors used on neighbors in $Y$ from the lists outside $X\cup Y$.
The resulting assignment $L'\subseteq L$ satisfies
$|L'|\ge|L|-3q-d\ge|L|-C(t+d+1)$.
Any $L'$-coloring would extend to an $L$-coloring of $G$, so $L'$ is bad.
\end{proof}

Delcourt and Postle~\cite{DelcourtPostle2025} proved the following result as Theorem~2.3
in their paper.

\begin{lemma}[\cite{DelcourtPostle2025}]\label{thm:small-connected}
There is an absolute constant $C=\cst{thm:small-connected}\ge1$ such that,
for integers $k\ge t\ge3$, every $K_t$-minor-free graph $G$ with
$\mathsf{d}(G)\ge Ck$ contains a $k$-connected
subgraph on at most $C^2t(\ln t)^3$ vertices.
\end{lemma}

Adding edges preserves vertex connectivity, so this subgraph may be taken
to be induced. The following list version of connectivity extraction was
extracted by Postle~\cite{Postle2020} (Theorem~3.16) from the argument
of Gir\~ao and Narayanan~\cite{GiraoNarayanan2022}.

\begin{lemma}[\cite{GiraoNarayanan2022}]\label{thm:bad-connected}
Let $k\ge1$ be an integer. If $L$ is a bad list assignment on $G$
with $|L|\ge4k+1$, then $G$ has an induced $k$-connected subgraph $H$
with a bad sublist $L_H\subseteq L$ satisfying $|L_H|\ge|L|-4k$.
\end{lemma}

The next lemma follows the proof of Corollary~4.12 in~\cite{Postle2020}.
Enlarge $\tst{thm:local-list}$ so that $2\cst{thm:small-connected}^2\le\ln t$
for every $t\ge \tst{thm:local-list}$.

\begin{lemma}\label{cor:many-pieces}
For integers $k\ge t\ge \tst{thm:local-list}$, every $K_t$-minor-free graph $G$
with a bad list assignment $L$ satisfying
$|L|\ge8\cst{lem:deletion}\cst{thm:small-connected}k$ contains
$\ceil{\ln t}+1$ disjoint induced $k$-connected subgraphs,
each on at most $\cst{thm:small-connected}^2t(\ln t)^3$ vertices.
\end{lemma}

\begin{proof}
Choose a maximal family $\mathcal H$ of disjoint induced $k$-connected
subgraphs, each of order at most $\cst{thm:small-connected}^2t(\ln t)^3$.
Suppose that $|\mathcal H|\le\ceil{\ln t}$, and let
$X=\bigcup_{H\in\mathcal H}V(H)$. Then
$|X|\le2\cst{thm:small-connected}^2t(\ln t)^4\le t(\ln t)^5$.
Since $\cst{lem:deletion}(t+1)\le2\cst{lem:deletion}k\le|L|/2$,
by Lemma~\ref{lem:deletion} with $Y=\varnothing$, the graph
$G\setminus X$ has a bad sublist $L'\subseteq L$ with $|L'|\ge|L|/2$.
By Observation~\ref{obs:bad-density}, it has an induced subgraph of
density at least $|L'|/2\ge|L|/4\ge\cst{thm:small-connected}k$.
By Lemma~\ref{thm:small-connected}, this subgraph contains another
induced $k$-connected subgraph of the required order, contradicting
the maximality of $\mathcal H$.
\end{proof}

\section{Paths and rooted minors}\label{sec:fans}

In this section, we collect some useful lemmas. Most of the lemmas
are from~\cite{Postle2020,DelcourtPostle2025}.
%We slightly modify some statements
%for use in our proof.

Given a graph $G$ and disjoint sets $A,B\subseteq V(G)$,
an \emph{$(A,B)$-path} in $G$ is a path with one end in $A$, the other
in $B$, and no internal vertex in $A\cup B$. A set $X\subseteq V(G)$
\emph{separates} $A$ from $B$ in $G$ if every $(A,B)$-path in $G$ contains
a vertex of $X$. A separating set may contain
vertices of $A\cup B$.
We use the following form of Menger's theorem~\cite{Menger1927} to find disjoint paths.

\begin{lemma}[Menger's theorem~\cite{Menger1927}]\label{lem:menger}
For disjoint vertex sets $A,B$ in a graph, the maximum number of
vertex-disjoint $(A,B)$-paths equals the minimum size of a set
separating $A$ from $B$.
\end{lemma}

We will also need two paths from each vertex of a prescribed set.
Given a graph $G$ and disjoint sets $A,B\subseteq V(G)$,
a \emph{double fan from $A$ to $B$}
is a family $\calP$ of $(A,B)$-paths in $G$ such that every vertex
of $A$ is an end of exactly two paths in $\calP$ and
$V(P)\cap V(Q)\subseteq A$ for all distinct $P,Q\in\calP$.
The following lemma is a consequence of Menger's theorem~\cite{Menger1927}.

\begin{lemma}\label{lem:double-fan}
Let $A,B$ be disjoint vertex sets in a graph $G$, with $A\ne\varnothing$.
If $G$ is $2|A|$-connected and $|B|\ge2|A|$, then $G$ contains
a double fan from $A$ to $B$.
\end{lemma}

\begin{proof}
For each $a\in A$, add a copy $a'$ with the same neighbors as $a$.
Let $A'$ be the set of copies. The resulting graph is still
$2|A|$-connected. By Lemma~\ref{lem:menger}, it contains $2|A|$ disjoint
$(A\cup A',B)$-paths, one with an end at each vertex of $A\cup A'$.
Replace each end $a'$ by $a$. The resulting paths form the required double fan.
\end{proof}

The next lemma follows from the redundancy version of Menger's theorem
stated by Delcourt and Postle~\cite{DelcourtPostle2025} (Lemma~5.3)
and Postle~\cite{Postle2020} (Lemma~3.17).
In their formulation, both path families are double fans. We record
the reduction for the form needed here.

\Needspace{7\baselineskip}
\begin{lemma}\label{lem:join-fans}
Let $r\ge1$ be an integer, and let $A,B,C$ be pairwise disjoint vertex sets
in a graph $G$.
Let $\calP$ be a double fan from $A$ to $C$ avoiding $B$.
Let $\calQ$ be a family of $2r$ disjoint $(B,C)$-paths in $G\setminus A$.
The subgraph formed by the paths in $\calP$ and $\calQ$ contains
$|A|+r$ disjoint $(A\cup B,C)$-paths.
Each vertex of $A$ is an end of exactly one of these paths.
The other $r$ paths have distinct ends among the ends of $\calQ$ in $B$.
\end{lemma}

\begin{proof}
If $A=\varnothing$, any $r$ paths of $\calQ$ suffice. Assume $A\ne\varnothing$
and work in the subgraph formed by $\calP$ and $\calQ$.
Pair the $2r$ ends of $\calQ$ in $B$. For each pair, add a new vertex
adjacent to both ends, and let $D$ be the set of these $r$ new vertices.
The extended paths form a double fan from $D$ to $C$ avoiding $A$.
By the cited redundancy version of Menger's theorem, there are
$|A|+r$ disjoint $(A\cup D,C)$-paths, one starting at each vertex of $A\cup D$.
Delete the end in $D$ from each path starting there. The shortened paths
have distinct ends among the ends of $\calQ$ in $B$.
Every vertex of $B$ has degree at most one in the original subgraph,
so none is an internal vertex of the resulting paths.
\end{proof}

The next lemma follows from Lemma~4.14 of Postle~\cite{Postle2020} and concerns
the sparsity of path families of minimum total length.

\Needspace{7\baselineskip}
\begin{lemma}[\cite{Postle2020}]\label{lem:short-fan}
Let $\ell\ge1$ be an integer, and let $A,B$ be disjoint vertex sets
in a graph $G$. Let $\calP=\{P_i\}_{i\in[\ell]}$ be a family of $\ell$
disjoint $(A,B)$-paths with minimum total length among all such families.
Then the following hold.
\begin{enumerate}
\item For $i\in[\ell]$ and every $v\in V(G)\setminus V(\calP)$,
$|N_G(v)\cap V(P_i)|\le3$.
\item $G[V(\calP)\setminus(A\cup B)]$ is $(2\ell-1)$-degenerate.
\end{enumerate}
\end{lemma}

\begin{remark}\label{rem:short-double-fan}
The conclusions of Lemma~\ref{lem:short-fan} also hold for a double fan
from $A$ to $B$ of minimum total length, with $\ell=2|A|\ge1$.
Split each vertex of $A$ into two copies with the same neighbors outside
$A$, one for each path. We obtain $\ell$ disjoint paths, and minimum
total length is preserved. A shorter family of $\ell$ disjoint paths
from the copies to $B$ would form a shorter double fan after replacing
the copies by the original vertices, a contradiction. The graph induced
by the internal vertices and the number of neighbors of each outside
vertex on each path are unchanged.
\end{remark}

We recall the definitions of linkages and linked graphs.

\begin{definition}\label{def:linkage}
Let $G$ be a graph, let $k\ge0$ be an integer, and let
$\mathcal S=\{(s_i,t_i)\}_{i\in[k]}$ be a set of pairs of vertices
of $G$ such that $s_i\ne s_j$, $t_i\ne t_j$ and $s_i\ne t_j$
whenever $i\ne j$.
An \emph{$\mathcal S$-linkage} in $G$ is a family
$\calP=\{P_i\}_{i\in[k]}$ of pairwise vertex-disjoint paths in $G$
such that $P_i$ has ends $s_i$ and $t_i$ for every $i\in[k]$.
If $s_i=t_i$, the path $P_i$ consists of the single vertex $s_i$.
For $k=0$, both $\mathcal S$ and $\calP$ are empty.
\end{definition}

\begin{definition}\label{def:linked}
For an integer $k\ge1$, a graph $G$ is \emph{$k$-linked} if
$\mathsf{v}(G)\ge2k$ and,
for every set $\mathcal S=\{(s_i,t_i)\}_{i\in[k]}$ of pairs of vertices of $G$
with $s_i\ne s_j$, $t_i\ne t_j$ and $s_i\ne t_j$ whenever $i\ne j$,
$G$ contains an $\mathcal S$-linkage.
\end{definition}

Thomas and Wollan~\cite{ThomasWollan2005} improved the connectivity bound
of Bollob\'as and Thomason~\cite{BollobasThomason1996} as follows.

\begin{lemma}[\cite{ThomasWollan2005}]\label{thm:linked}
For every integer $k\ge1$, every $10k$-connected graph is $k$-linked.
\end{lemma}

\begin{definition}\label{def:model}\label{def:rooted-model}
For graphs $H$ and $G$, an \emph{$H$ model in $G$} is a collection
$\calM=\{B_v:v\in V(H)\}$ of pairwise vertex-disjoint nonempty
connected subgraphs of $G$ such that, for every $uv\in E(H)$,
some edge of $G$ has one end in $V(B_u)$ and the other in $V(B_v)$.
The subgraphs $B_v$ are the \emph{branch sets} of $\calM$.
Let $R\subseteq V(G)$ have size $\mathsf{v}(H)$.
The model $\calM$ is \emph{rooted at $R$} if
$|V(B_v)\cap R|=1$ for every $v\in V(H)$.
\end{definition}

A graph $G$ contains an $H$ minor if and only if $G$ contains an $H$ model.
The following rooted-minor result is due to Kawarabayashi~\cite{Kawarabayashi2007}.

\begin{lemma}[\cite{Kawarabayashi2007}]\label{lem:rooted-minor}
Let $p\ge1$ be an integer. If a $p$-connected graph $G$ contains
a $K_{2p}$ minor, then for every set $R\subseteq V(G)$ of size $p$,
$G$ contains a $K_p$ model rooted at $R$.
\end{lemma}

The following definition combines a rooted complete model with a linkage.

\begin{definition}\label{def:woven}
Let $p\ge1$ and $q\ge0$ be integers. A graph $G$ with at least
$p+2q$ vertices is \emph{$(p,q)$-woven} if the following holds.
For every three sets $R=\{r_1,\ldots,r_p\}$,
$S=\{s_1,\ldots,s_q\}$ and $T=\{t_1,\ldots,t_q\}$ in $V(G)$
with $|R|=p$, $|S|=|T|=q$ and $s_i\ne t_j$ whenever $i\ne j$,
there are a $K_p$ model $\calM$ in $G$ rooted at $R$ and a
$\{(s_i,t_i)\}_{i\in[q]}$-linkage $\calP$ in $G$ such that
\[
 V(\calM)\cap V(\calP)=R\cap(S\cup T).
\]
\end{definition}

The set $R$ may meet $S\cup T$. When $q=0$, the path family $\calP$ is empty.
Delcourt and Postle~\cite{DelcourtPostle2025} proved the following result as Lemma~5.13
in their paper.

\begin{lemma}[\cite{DelcourtPostle2025}]\label{thm:routing}
There is an absolute integer
$C=\cst{thm:routing}\ge\max\{10,\cst{thm:knit}\}$ such that,
for all integers $p,q\ge2$, every graph $G$ with
$\kappa(G)\ge C(p\sqrt{\ln p}+q)$ is $(p,q)$-woven.
\end{lemma}

The following lemma of Delcourt and Postle~\cite{DelcourtPostle2025} (Lemma~5.15 in their paper)
allows us to find a rooted $K_p$ model in a $(p,q)$-woven subgraph and
reroute a given linkage to avoid the model except at the roots.

\begin{lemma}[\cite{DelcourtPostle2025}]\label{lem:reroute}
Let $p,q\ge1$ be integers, and let $G$ be a graph.
Let $S=\{s_1,\ldots,s_q\}$ and $T=\{t_1,\ldots,t_q\}$ be disjoint
sets of $q$ vertices in $G$, and let $\calP$ be an
$\{(s_i,t_i)\}_{i\in[q]}$-linkage in $G$.
Suppose that $H$ is a $(p,q)$-woven subgraph of $G$ and
$R=\{r_1,\ldots,r_p\}\subseteq V(H)\cap(S\cup T)$ has size $p$.
Then there are an $\{(s_i,t_i)\}_{i\in[q]}$-linkage $\calP'$ in $G$
and a $K_p$ model $\calM$ in $H$ rooted at $R$ such that
$V(\calP')\cap V(\calM)=R$ and $V(\calP')\subseteq V(H)\cup V(\calP)$.
\end{lemma}

\section{A separation lemma}\label{sec:separation}

In this section, we show how to find two vertex-disjoint subgraphs
with large bad sublists of a given bad list assignment.

\begin{definition}\label{def:separable}
Let $k\ge0$ be an integer, let $G$ be a graph, and let $L$ be a bad
list assignment on $G$.
We say that $L$ is \emph{$k$-separable in $G$}
if there are vertex-disjoint subgraphs $H_1,H_2$ of $G$ with bad sublist
assignments $L_1,L_2\subseteq L$, respectively, satisfying
$|L_i|\ge|L|-k$ for $i\in[2]$.
\end{definition}

We next show that a sufficiently large bad list assignment on a
$K_t$-minor-free graph $G$ is $k$-separable in $G$ for some $k=O(t)$.
The construction follows the proof of Lemma~5.4 in~\cite{Postle2020}.

After increasing $\tst{thm:local-list}$ if necessary, the following holds.

\begin{lemma}\label{lem:separation}
There is an absolute constant $C=\cst{lem:separation}$ such that,
for every integer $t\ge \tst{thm:local-list}$, every bad list assignment $L$ on a
$K_t$-minor-free graph $G$ with $|L|\ge2Ct$
is $Ct$-separable in $G$.
\end{lemma}

\begin{proof}
\phantomsection\label{par:separation-constants}
We take
\begin{equation}\label{eq:separation-parameters}
 N_0=\ceil{100\cst{thm:small-connected}^2\cst{lem:bipartite}\cst{thm:routing}},\quad
 C=100\cst{lem:deletion}N_0^2.
\end{equation}
With $N_0$ fixed, increase the threshold $\tst{thm:local-list}$ from
Lemma~\ref{thm:local-list} so that $\ln t\ge16N_0^2$
for every $t\ge \tst{thm:local-list}$.

Let $k=Ct$, and suppose that $L$ is not
$k$-separable in $G$. By hypothesis, $|L|\ge2k$.
Since $N_0\ge1000$ and $\cst{lem:deletion}\ge1$,
\begin{equation}\label{eq:separation-list-budget}
 k=100\cst{lem:deletion}N_0^2t\ge16N_0t.
\end{equation}

Let $p=\ceil{\sqrt{\ln(2t)}}$ and $q=\ceil{t/p}$.
Since $t\ge3$, we have $p\le\min\{2\ln t,t\}$, $q\le t$ and
$t\le pq<t+p\le2t$.

For $r=0,\ldots,p$, we recursively construct a tuple
\[
 (H_r,L_r,\calM_r,\calT_r^1,\calT_r^2,U_r)
\]
satisfying the following conditions.
\begin{enumerate}[label=(C\arabic*),ref=(C\arabic*),leftmargin=3em]
\item\label{inv:carrier}
$H_r$ is an induced $N_0t$-connected subgraph of $G$ with a bad
sublist $L_r\subseteq L$ satisfying $|L_r|\ge|L|-k/2$.
\item\label{inv:tangent}
$\calM_r=(B_r^1,\ldots,B_r^{rq})$ is a $K_{rq}$ model in $G$.
For each $i\in[rq]$, $H_r$ and $B_r^i$ have exactly one common vertex, denoted by $w_r^i$.
Let $W_r=\{w_r^1,\ldots,w_r^{rq}\}$; for $r=0$, the model and $W_0$ are empty.
\item\label{inv:paths}
$\calT_r^1$ and $\calT_r^2$ are path families in $G$ avoiding $H_r$.
The paths in each family are pairwise disjoint, and
$|\calT_r^1|+|\calT_r^2|\le16tr$.
\item\label{inv:size}
$U_r\subseteq V(G)$ satisfies $|U_r|\le t+rN_0t(\ln t)^4$,
$W_r\subseteq U_r$ and $V(\calM_r)\subseteq U_r\cup V(\calT_r^1\cup\calT_r^2)$.
\item\label{inv:degenerate}
$G[V(\calT_r^1\cup\calT_r^2)\setminus U_r]$ is $100t$-degenerate.
\item\label{inv:neighbors}
Every vertex outside $U_r\cup V(\calT_r^1\cup\calT_r^2)$ has at most three neighbors on
each path in $\calT_r^1\cup\calT_r^2$ and at most $12t$ neighbors in $V(\calT_r^1\cup\calT_r^2)$.
\end{enumerate}

By Inequality~\eqref{eq:separation-list-budget}, $|L|\ge2k\ge4N_0t+1$.
By Lemma~\ref{thm:bad-connected} (with $N_0t$ in place of $k$), there are an
induced $N_0t$-connected subgraph $H_0$ and a bad sublist
$L_0\subseteq L$ with $|L_0|\ge|L|-4N_0t\ge|L|-k/2$.
Let $\calM_0$ be the empty model, let $U_0=W_0=\varnothing$, and let $\calT_0^1=\calT_0^2=\varnothing$.
Then~\ref{inv:carrier}--\ref{inv:neighbors} hold.

Let $0\le r<p$, and suppose that the tuple at stage $r$
has been constructed and satisfies~\ref{inv:carrier}--\ref{inv:neighbors}.
We construct the tuple at stage $r+1$ in six steps
(see Fig.~\ref{fig:separation-step} for an illustration).

\medskip
\noindent\textbf{Step 1.} Reserving subgraphs and paths to extend $\calM_r$ to $\calM_{r+1}$.\par\nobreak
\medskip\nobreak
To form $\calM_{r+1}$, we enlarge the existing $rq$ branch sets and add $q$ new branch sets.
Since $|L_r|\ge3k/2>2t\ge|W_r|$, color $W_r$ with distinct colors
and delete it, removing the colors used on neighbors in $W_r$.
The resulting bad sublist on $H_r\setminus W_r$ has size at least $|L_r|-rq$.
By $|L|\ge2k$, \ref{inv:carrier} and Inequality~\eqref{eq:separation-list-budget}, its size is at least
\[
 |L_r|-rq\ge\frac{3k}{2}-2t\ge k
 \ge8\cst{lem:deletion}\cst{thm:small-connected}N_0t.
\]
Since $p\le\ceil{\ln t}$, Lemma~\ref{cor:many-pieces} (with $N_0t$ in place of $k$) gives pairwise disjoint induced $N_0t$-connected subgraphs $J_1,\ldots,J_p,D$ of $H_r\setminus W_r$ satisfying
\begin{equation}\label{eq:separation-small-pieces}
 \max\{\mathsf v(D),\mathsf v(J_1),\ldots,\mathsf v(J_p)\}
 \le\cst{thm:small-connected}^2t(\ln t)^3.
\end{equation}
Since $\mathsf v(J_i)>N_0t\ge2q$, choose a set $R_i$ of $2q$ vertices in $J_i$ for each $i\in[p]$.
Let $R=W_r\cup\bigcup_{i=1}^pR_i$, so $|R|=rq+2pq\le6t$.
Let
\[
 G'=G\setminus\bigl((V(\calM_r)\setminus W_r)\cup V(\calT_r^1\cup\calT_r^2)\bigr).
\]
By~\ref{inv:tangent} and~\ref{inv:paths}, $H_r\subseteq G'$.
Recall that $D$ is $N_0t$-connected, so $\mathsf v(D)>N_0t\ge2|R|$.
Also, by~\ref{inv:carrier}, $\kappa(H_r)\ge N_0t\ge2|R|$.
By Lemma~\ref{lem:double-fan} in $H_r$ with $A=R$ and $B=V(D)$, there is a double fan from $R$ to $V(D)$ in $G'$.
Choose such a double fan $\calP_1$ with minimum total length in $G'$.
We have $|\calP_1|=2|R|=2rq+4pq\le12t$.

The subgraphs $J_1,\ldots,J_p$ will supply parts of the new branch sets and the enlarged old branch sets.
The graph $D$ and the paths in $\calP_1$ are reserved for joining the parts of these branches.

\medskip
\noindent\textbf{Step 2.} Constructing $H_{r+1}$ and its bad sublist $L_{r+1}$.\par\nobreak
\medskip\nobreak
Let
\[
 X=U_r\cup V(D)\cup\bigcup_{i=1}^pV(J_i),\quad
 Y=V(\calT_r^1\cup\calT_r^2\cup\calP_1)\setminus X.
\]
We verify the conditions of Lemma~\ref{lem:deletion} with $d=100t$.

Recall that $\calP_1$ has minimum total length in $G'$,
$|\calP_1|\le12t$, and $R\cup V(D)\subseteq X$.
By Remark~\ref{rem:short-double-fan},
$G[V(\calP_1)\setminus X]$ is $(24t-1)$-degenerate.
The paths in $\calP_1$ avoid $V(\calT_r^1\cup\calT_r^2)$.
By~\ref{inv:neighbors}, every vertex of $V(\calP_1)\setminus X$
has at most $12t$ neighbors on these old paths.
Delete the vertices of $V(\calP_1)\setminus X$ first in a degeneracy
ordering; each has at most $36t-1$ remaining neighbors in $G[Y]$.
The remaining graph is $100t$-degenerate by~\ref{inv:degenerate},
since $U_r\subseteq X$. Thus $G[Y]$ is $100t$-degenerate.

By~\ref{inv:size} at stage $r$,
$V(\calM_r)\cup V(\calT_r^1\cup\calT_r^2)\subseteq X\cup Y$,
so every vertex outside $X\cup Y$ belongs to $G'$.
By~\ref{inv:neighbors} at stage $r$ and
Remark~\ref{rem:short-double-fan} in $G'$,
every vertex outside $X\cup Y$ has at most $12t+3|\calP_1|\le48t<100t$
neighbors in $Y$.

Note that $p+1\le3\ln t$, $r+1\le p\le2\ln t$ and
$N_0\ge3\cst{thm:small-connected}^2$.
By~\ref{inv:size} at stage $r$ and Inequality~\eqref{eq:separation-small-pieces}, we have
\[
 |X|\le t+rN_0t(\ln t)^4+(p+1)\cst{thm:small-connected}^2t(\ln t)^3
 \le t+(r+1)N_0t(\ln t)^4\le t+pN_0t(\ln t)^4.
\]
By the choice of $\tst{thm:local-list}$, $p\le2\sqrt{\ln t}$ and
$N_0p\le(\ln t)/2$, so $|X|\le t+\tfrac12t(\ln t)^5\le t(\ln t)^5$.
By Inequality~\eqref{eq:separation-list-budget},
\[
 \cst{lem:deletion}(101t+1)
 \le102\cst{lem:deletion}t
 \le2\cst{lem:deletion}N_0t\le\frac k4<|L|.
\]
By Lemma~\ref{lem:deletion} (with $d=100t$),
$G\setminus(X\cup Y)$ has a bad sublist of $L$ of size at least
$|L|-k/4\ge7k/4\ge4N_0t+1$.
By Lemma~\ref{thm:bad-connected} (with $N_0t$ in place of $k$) and
Inequality~\eqref{eq:separation-list-budget},
$G\setminus(X\cup Y)$ has an induced $N_0t$-connected subgraph $H_{r+1}$
with a bad sublist $L_{r+1}\subseteq L$ satisfying
\[
 |L_{r+1}|\ge|L|-\frac k4-4N_0t\ge|L|-\frac k2.
\]

The subgraph $H_{r+1}$ and the bad sublist $L_{r+1}$ satisfy~\ref{inv:carrier}.

\medskip
\noindent\textbf{Step 3.} Joining $R$ and $H_{r+1}$ through $D$.\par\nobreak
\medskip\nobreak
Suppose that $|V(H_r)\cap V(H_{r+1})|\le N_0t$.
Since $|L_{r+1}|\ge3k/2>N_0t$, color and delete this intersection
as in Section~\ref{subsec:bad-lists-connectivity}.
The remaining bad sublist on $H_{r+1}\setminus V(H_r)$ has size at least
\[
 |L_{r+1}|-|V(H_r)\cap V(H_{r+1})|\ge|L|-\frac k2-N_0t\ge|L|-k.
\]
By~\ref{inv:carrier} at stage $r$, $H_r$ also has a bad sublist of $L$
of size at least $|L|-k$.
The subgraphs $H_r$ and $H_{r+1}\setminus V(H_r)$ are disjoint,
so $L$ is $k$-separable in $G$, a contradiction.
Hence $|V(H_r)\cap V(H_{r+1})|>N_0t$.

The disjoint sets $V(D)$ and $V(H_r)\cap V(H_{r+1})$ lie in $H_r\setminus R$ and each has more than $N_0t\ge2(r+1)q$ vertices.
Recall from Step 1 that $|R|\le6t$.
By~\ref{inv:carrier} at stage $r$ and the choice of $N_0$ in Equation~\eqref{eq:separation-parameters},
\[
 \kappa(H_r\setminus R)\ge\kappa(H_r)-|R|\ge(N_0-6)t\ge2(r+1)q.
\]
By Lemma~\ref{lem:menger}, $H_r\setminus R$ contains $2(r+1)q$ disjoint $(V(H_r)\cap V(H_{r+1}),V(D))$-paths.
Since $H_r\subseteq G'$, we choose $\calP_2$ with minimum total length among all families of $2(r+1)q$ disjoint $(V(H_{r+1}),V(D))$-paths in $G'\setminus R$.

Since $H_{r+1}\subseteq G\setminus(X\cup Y)$, the double fan $\calP_1$ avoids $H_{r+1}$.
The paths in $\calP_2$ avoid $R$.
By Lemma~\ref{lem:join-fans} (applied to the sets $R,V(H_{r+1}),V(D)$, with $(r+1)q$ in place of the parameter $r$), the subgraph formed by $\calP_1$ and $\calP_2$ contains a family $\calQ_0$ of
$|R|+(r+1)q=(2r+1)q+2pq\le8t$ disjoint $(R\cup V(H_{r+1}),V(D))$-paths.
Each vertex of $R$ is an end of one path in $\calQ_0$.
The remaining $(r+1)q$ paths in $\calQ_0$ have distinct ends $w_{r+1}^1,\ldots,w_{r+1}^{(r+1)q}$ in $V(\calP_2)\cap V(H_{r+1})$.
Let $W_{r+1}=\{w_{r+1}^1,\ldots,w_{r+1}^{(r+1)q}\}$.
No path in $\calQ_0$ has an internal vertex in $R\cup V(H_{r+1})\cup V(D)$.
Let $\mathcal S$ be the set of pairs of ends of the paths in $\calQ_0$,
with the first entry in $R\cup W_{r+1}$ and the second in $V(D)$.
Then $\calQ_0$ is an $\mathcal S$-linkage.

\begin{figure}[!htb]
\centering
\captionsetup{skip=6pt}
\begingroup
% Inline TikZ for the separation construction; synchronized with separation_step.tex.
\definecolor{branchBlue}{HTML}{1479F5}
\definecolor{branchMagenta}{HTML}{E72C9B}
\definecolor{blueFill}{HTML}{C5DCFC}
\colorlet{magentaFill}{branchMagenta!25}
\definecolor{frameGray}{HTML}{B6C0C9}
\definecolor{Xshade}{HTML}{FFF9E3}
\definecolor{Rshade}{HTML}{11B45B}
\definecolor{newVertexFill}{HTML}{8B45E6}
\definecolor{TOneShade}{HTML}{BDBDBD}
\definecolor{TTwoShade}{HTML}{E5BD53}
\definecolor{auxGray}{HTML}{5E6871}
\definecolor{inactiveBranch}{HTML}{D9D9D9}
\resizebox{\linewidth}{!}{%
\begin{tikzpicture}[x=.6cm,y=-.6cm,
 every node/.style={font=\fontsize{11}{12.5}\selectfont,text=black,inner sep=0pt,outer sep=0pt},
 region/.style={draw=frameGray,line width=.45pt},
 part/.style={draw=black!35,line width=.4pt},
 vertex/.style={circle,fill=white,draw=black,line width=.45pt,inner sep=0pt,minimum size=4pt},
 route/.style={line width=1.5pt,line cap=round,line join=round},
 aux/.style={draw=auxGray,line width=.8pt,dash pattern=on 2.6pt off 1.5pt,line cap=round},
 track one/.style={draw=TOneShade,draw opacity=.58,line width=5.4pt,line cap=round,line join=round},
 track two/.style={draw=TTwoShade,draw opacity=.58,line width=5.4pt,line cap=round,line join=round},
 internal/.style={line width=.9pt,line cap=round,line join=round},
 edge/.style={draw=black!38,line width=.5pt},
 tag/.style={fill=none,inner sep=1.1pt}]
% Natural TikZ bounds allow the visible drawing to use the full text width.

% Preserve the height; widen D and the three J_k while narrowing the right carrier.
\begin{scope}[xshift=-1.08cm]
% The two carriers overlap, while H_{r+1} avoids D, all J_k, and the old model.
\path[fill=black!1] (8.1,2.5) rectangle (36.6,22.2);
\node[tag] (sepHprev) at (14.66,2.5) {$H_r$};
\draw[region] (sepHprev.west)--(8.1,2.5)--(8.1,22.2)--(36.6,22.2)--(36.6,2.5)--(sepHprev.east);
\path[fill=white,fill opacity=.16] (35.7,3.0) rectangle (39.9,16.5);
\node[tag] (sepHnext) at (37.8,3.0) {$H_{r+1}$};
\draw[region] (sepHnext.west)--(35.7,3.0)--(35.7,16.5)--(39.9,16.5)--(39.9,3.0)--(sepHnext.east);

% The old complete model is represented by three pairwise adjacent branches.
% The three connecting edges run outside the colored branch interiors.
\node[tag] (sepMprev) at (5.75,3.6) {$\mathcal M_r$};
\draw[region,dashed] (sepMprev.west)--(3.45,3.6)--(3.45,14.85)--(8.35,14.85)--(8.35,3.6)--(sepMprev.east);
% Use distinct boundary attachments and separated arcs for the three adjacencies.
\pgfmathsetmacro{\sepAdjacentX}{6.75-1.35*sqrt(1-(.8/1.60)^2)}
\pgfmathsetmacro{\sepOuterX}{6.75-1.35*sqrt(1-(.5/1.60)^2)}
\draw[edge] (\sepAdjacentX,6.3)
 .. controls (4.70,6.75) and (4.70,7.75) .. (\sepAdjacentX,8.2);
\draw[edge] (\sepAdjacentX,9.8)
 .. controls (4.70,10.25) and (4.70,11.25) .. (\sepAdjacentX,11.7);
\draw[edge] (\sepOuterX,5.0)
 .. controls (3.65,5.35) and (3.65,12.65) .. (\sepOuterX,13.0);
\foreach \yy/\idx/\col in {5.5/1/inactiveBranch,9.0/i/blueFill,12.5/{rq}/inactiveBranch}{
 \fill[\col] (6.75,\yy) ellipse (1.35 and 1.60);
 \node at (6.65,\yy) {$B_r^{\idx}$};}
\foreach \yy in {7.25,10.75}{\node at (6.75,\yy) {$\vdots$};}
% The B branches touch the carrier at exactly the displayed W vertices.
\path[fill=Rshade,fill opacity=.25] (7.78,3.05) rectangle (9.45,14.3);
\node[tag] (sepWprev) at (9.45,3.8) {$W_r$};
\draw[part] (sepWprev.north)--(9.45,3.05)--(7.78,3.05)--(7.78,14.3)--(9.45,14.3)--(sepWprev.south);
\node[tag] at (8.64,10.0) {$w_r^{i}$};

% D is horizontal. The two colored connected pieces are representative.
\path[fill=Xshade] (11.4643,3.9) rectangle (31.4379,10.8);
\node[tag] (sepD) at (12.4629,3.9) {$D$};
\draw[region] (sepD.west)--(11.4643,3.9)--(11.4643,10.8)--(31.4379,10.8)--(31.4379,3.9)--(sepD.east);
% Single smooth silhouettes taper continuously to the boundary attachments.
% Both the old and the new branch have a root in the next carrier H_{r+1}.
\path[fill=blueFill,draw=none]
 (11.4643,7.35)
 .. controls (12.6475,7.16) and (13.9125,5.78) .. (15.4589,5.78)
 .. controls (17.075,5.78) and (18.225,5.65) .. (21.79,5.35)
 .. controls (25.125,5.10) and (29.15,4.78) .. (31.4379,4.78)
 -- (31.4379,5.22)
 .. controls (29.15,5.27) and (23.86,6.28) .. (19.49,6.85)
 .. controls (16.615,7.07) and (15.235,7.25) .. (14.269,8.33)
 .. controls (13.694,8.98) and (13.7745,9.95) .. (13.7491,10.8)
 -- (13.1741,10.8)
 .. controls (13.165,9.40) and (13.05,8.92) .. (13.4295,8.34)
 .. controls (13.6135,8.06) and (12.314,8.01) .. (11.4643,7.85)
 -- cycle;
\path[fill=magentaFill,draw=none]
 (17.302,10.8)
 .. controls (17.42,9.85) and (19.605,9.03) .. (21.79,8.42)
 .. controls (24.32,7.73) and (26.8,7.05) .. (28.92,6.69)
 .. controls (29.84,6.53) and (30.8,6.57) .. (31.4379,6.57)
 -- (31.4379,7.03)
 .. controls (30.45,7.0) and (28.66,7.57) .. (27.1,8.39)
 .. controls (28.19,8.92) and (29.5525,9.42) .. (29.728,10.8)
 -- (29.153,10.8)
 .. controls (28.92,9.90) and (26.7695,9.66) .. (25.2745,9.84)
 .. controls (24.3545,9.95) and (23.745,10.30) .. (23.7359,10.8)
 -- (23.1609,10.8)
 .. controls (23.1585,10.15) and (23.4575,9.68) .. (23.2735,9.59)
 .. controls (22.135,9.62) and (19.26,9.69) .. (17.877,10.8)
 -- cycle;
\draw[internal,branchBlue] (11.4643,7.6)
 .. controls (12.7957,7.55) and (14.1273,7.05) .. (15.4589,6.8)
 .. controls (19.72,6.57) and (26.5111,5.0) .. (31.4379,5.0);
\draw[internal,branchBlue] (13.4616,10.8)
 .. controls (13.4616,9.2) and (13.1952,8.4) .. (15.4589,6.8);
\draw[internal,branchMagenta] (17.5895,10.8)
 .. controls (18.6548,9.35) and (21.5843,9.6) .. (24.1143,9.25);
\draw[internal,branchMagenta] (23.4484,10.8)
 .. controls (23.1821,10.35) and (24.3139,9.65) .. (24.1143,9.25);
\draw[internal,branchMagenta] (29.4405,10.8)
 .. controls (29.1743,9.6) and (26.5111,9.3) .. (24.1143,9.25);
\draw[internal,branchMagenta] (24.1143,9.25)
 .. controls (26.0,8.16) and (29.44,6.8) .. (31.4379,6.8);
\node at (15.4589,6.12) {$D_i$};
\node at (23.50,8.50) {$D_{rq+j}$};
\node at (20.7852,4.60) {$\cdots$};

% Lower the local models inside the unchanged overall frame.
\begin{scope}[shift={(0,.65)}]
% Three representative J_k, their local models N_k, and their root sets R_k.
% In this representative configuration u_i belongs to R_1. Other roots and
% their routes are omitted; the two colored global branches are fully shown.
\foreach \left/\right/\k/\labelx in {9.6/17.5895/1/10.4655,
 18.7213/26.2448/2/19.5868,
 27.9093/34.9/p/33.8348}{
 \path[fill=Xshade] (\left,14.65) rectangle (\right,21.0);
 \node[tag] (sepJ) at (\labelx,14.65) {$J_{\k}$};
 \draw[region] (sepJ.west)--(\left,14.65)--(\left,21)--(\right,21)--(\right,14.65)--(sepJ.east);
 \node at ({(\left+\right)/2},20.55) {$K_{2q}$ model $\mathcal N_{\k}$};}
\node at (27.0837,18.5) {$\cdots$};
% Gray marks unused representatives; blue and magenta mark the two selected global branches.
% Four slender representative branch sets in each local complete model.
% Three roots in each inner box are v_k^1, v_k^j, v_k^q, in this order.
\foreach \xx/\rx/\col in {11.6639/.55/blueFill,
 13.55/.50/inactiveBranch,15.0/.50/magentaFill,16.45/.50/inactiveBranch,
 20.25/.50/inactiveBranch,21.7/.50/magentaFill,23.15/.50/inactiveBranch,24.9132/.50/inactiveBranch,
 29.45/.50/inactiveBranch,30.9/.50/magentaFill,32.35/.50/inactiveBranch,33.7016/.50/inactiveBranch}{
 \fill[\col] (\xx,18.45) ellipse [x radius=\rx,y radius=1.2];}
% All six adjacencies of each displayed four-branch model are shown.
\foreach \xa/\xb in {12.2139/13.05,14.05/14.50,15.50/15.95,
 20.75/21.20,22.20/22.65,23.65/24.4132,
 29.95/30.40,31.40/31.85,32.85/33.2016}{
 \draw[edge] (\xa,18.45)--(\xb,18.45);}
\foreach \xa/\xb/\xc/\xd in {11.6639/13.55/15.0/16.45,
 20.25/21.7/23.15/24.9132,
 29.45/30.9/32.35/33.7016}{
 \draw[edge] (\xa,19.65) .. controls (\xa,20.23) and (\xd,20.23) .. (\xd,19.65);
 \draw[edge] (\xa,19.65) .. controls (\xa,19.96) and (\xc,19.96) .. (\xc,19.65);
 \draw[edge,preaction={draw=white,line width=1.3pt}]
  (\xb,19.65) .. controls (\xb,19.96) and (\xd,19.96) .. (\xd,19.65);}
\foreach \left/\right/\k/\labelx in {10.0661/17.0568/1/10.7984,
 19.1873/25.7787/2/24.78,
 28.3752/34.4339/p/33.4352}{
 \path[fill=Rshade,fill opacity=.25] (\left,16.1) rectangle (\right,18.35);
 \node[tag] (sepR) at (\labelx,16.1) {$R_{\k}$};
 \draw[part] (sepR.west)--(\left,16.1)--(\left,18.35)--(\right,18.35)--(\right,16.1)--(sepR.east);
}
% The reserved roots are parallel; the selected j-th root is the middle one.
\foreach \left/\right/\k/\xa/\xj/\xq in {
 12.95/17.00/1/13.55/15.0/16.45,
 19.65/23.75/2/20.25/21.7/23.15,
 28.85/32.95/p/29.45/30.9/32.35}{
 \node[tag] (sepVfirst) at (\xa,17.85) {$v_{\k}^1$};
 \node[tag] (sepVmiddle) at (\xj,17.85) {$v_{\k}^j$};
 \node[tag] (sepVlast) at (\xq,17.85) {$v_{\k}^q$};
 \draw[part,draw=black!45,rounded corners=1.5pt]
  (sepVfirst.west)--(\left,17.85)--(\left,16.65)--(\right,16.65)--(\right,17.85)--(sepVlast.east);
 \draw[part,draw=black!45] (sepVfirst.east)--(sepVmiddle.west);
 \draw[part,draw=black!45] (sepVmiddle.east)--(sepVlast.west);
 \node at ({(\xa+\xj)/2},17.25) {{\scriptsize $\cdots$}};
 \node at ({(\xj+\xq)/2},17.25) {{\scriptsize $\cdots$}};
}
\node at (11.6639,18.05) {$u_i$};

\end{scope}

% All original P_2 ends in H_{r+1} lie in U_{r+1} minus X; W_{r+1} is only a subset.
% A pale-purple outer layer contains the new ends; the inner W frame has a light gray tint.
% The displayed ends lie midway across the carrier overlap; the narrower inner frame contains the retained ends.
\draw[draw=black!30,line width=.45pt,fill=newVertexFill,fill opacity=.30]
 (35.7,3.30) rectangle (39.55,11.1);
\path[fill=black,fill opacity=.06] (35.7,3.50) rectangle (38.5,8.25);
\node[anchor=west] at (36.80,3.95) {$w_{r+1}^1$};
\node[anchor=west] at (36.80,5.22) {$w_{r+1}^i$};
\node[anchor=west] at (36.80,6.8) {$w_{r+1}^{rq+j}$};
\node[tag] (sepWnext) at (38.5,5.9) {$W_{r+1}$};
\draw[draw=black!45,line width=.5pt] (sepWnext.north)--(38.5,3.50)--(35.7,3.50)--(35.7,8.25)--(38.5,8.25)--(sepWnext.south);
\node at (36.15,4.48) {$\vdots$};
\node at (36.15,6.12) {$\vdots$};
\node at (36.15,7.65) {$\vdots$};
\node at (36.15,9.75) {$\vdots$};

% Z_1 and Z_2 are sets, illustrated by several vertices and an ellipsis.
% They need not be disjoint in general; the representative drawing separates them.
\path[rounded corners=3pt,fill=newVertexFill,fill opacity=.30] (3.7,16.0) rectangle (7.55,20.05);
\node[tag] (sepZone) at (4.45,16.0) {$Z_1$};
\draw[part,rounded corners=3pt] (sepZone.west)--(3.7,16)--(3.7,20.05)--(7.55,20.05)--(7.55,16)--(sepZone.east);
\path[rounded corners=3pt,fill=newVertexFill,fill opacity=.30] (36.9,13.3) rectangle (39.55,19.9);
% Z_2 can meet both H_{r+1} and its complement; keep their overlap visible.
\draw[region] (36.9,16.5)--(39.55,16.5);
\node[tag] (sepZtwo) at (38.65,13.3) {$Z_2$};
\draw[part,rounded corners=3pt] (sepZtwo.west)--(36.9,13.3)--(36.9,19.9)--(39.55,19.9)--(39.55,13.3)--(sepZtwo.east);
% Full adjacencies are shown for the top vertex; other vertices have short continuations.
% Three separated routes leave Z_1 with distinct tangents and use different lanes.
\draw[edge] (5.9,16.7)
 .. controls (7.2,16.8) and (8.7,16.05) .. (9.9,15.05)
 .. controls (10.95,14.15) and (11.90,14.05) .. (12.2632,13.15);
\draw[edge] (5.9,16.7)
 .. controls (7.1,16.3) and (8.6,15.55) .. (9.8,14.72)
 .. controls (10.90,13.25) and (12.3,12.15) .. (13.5282,13.25);
\draw[edge] (5.9,16.7)
 .. controls (6.9,15.65) and (8.5,15.3) .. (9.9,14.2)
 .. controls (11.1,11.8) and (15.2,11.8) .. (17.19,13.0);
% The red arc counts distinct neighboring paths; it is not an edge.
\draw[draw=red!85!black,line width=.85pt] (5.9,16.7) ++(-56:1.05)
 arc[start angle=-56,end angle=4,radius=1.05];
\node[text=red!85!black] at (6.7,15.50) {$\geq 2t$};
% This unused Z_2 vertex lies in H_{r+1}; minimum total length forces its
% neighbors to be the first vertices outside H_{r+1} on the displayed P_2 paths.
\draw[edge] (38.3,14.2)
 .. controls (37.75,12.9) and (36.45,13.0) .. (35.3,11.8)
 .. controls (34.8334,10.6) and (35.1,8.4) .. (34.3673,6.8);
\draw[edge] (38.3,14.2)
 .. controls (37.55,14.1) and (36.3,13.8) .. (35.3,12.7)
 .. controls (34.9,11.9) and (34.9,10.0) .. (34.3673,8.7);
\draw[edge] (38.3,14.2)
 .. controls (37.5,14.5) and (36.3,14.7) .. (35.3,13.5)
 .. controls (34.6337,12.8) and (34.95,11.6) .. (34.3673,10.45);
\draw[draw=red!85!black,line width=.85pt] (38.3,14.2) ++(-201:.90)
 arc[start angle=-201,end angle=-112,radius=.90];
\node[text=red!85!black] at (37.8,12.55) {$\geq 2t$};
% Short uncapped stubs indicate omitted adjacencies at the remaining vertices.
\foreach \yy in {17.5,19.35}{
 \foreach \dy in {-.28,0,.28}{
  \draw[edge] (5.9,\yy) .. controls (6.2,\yy) and (6.45,{\yy+\dy}) .. (6.65,{\yy+\dy});
  \draw[edge,dash pattern=on .7pt off .8pt] (6.65,{\yy+\dy})--(6.98,{\yy+1.35*\dy});
 }}
\foreach \yy in {15.0,18.8}{
 \foreach \dy in {-.28,0,.28}{
  \draw[edge] (38.3,\yy) .. controls (38.1,\yy) and (37.94,{\yy+\dy}) .. (37.83,{\yy+\dy});
  \draw[edge,dash pattern=on .7pt off .8pt] (37.83,{\yy+\dy})--(37.6,{\yy+1.35*\dy});
 }}
% Further neighboring paths are indicated between the representative edges.
\node[rotate=-38] at (10.40,13.72) {{\scriptsize $\vdots$}};
\node at (36.5,13.81) {{\scriptsize $\vdots$}};
\node at (5.9,18.45) {$\vdots$};
\node at (38.3,17.6) {$\vdots$};

% Representative pairs of P_1 paths; the T underlay starts after each root.
% The other two displayed W roots also each have a pair of fan paths.
\def\pwtop{(9.8663,4.65)
 .. controls (10.29,4.3) and (10.75,4.98) .. (11.4643,4.6)}
\def\pwtopa{(9.8663,5.55)
 .. controls (10.3475,5.75) and (10.796,5.15) .. (11.4643,5.3)}
\def\pwbottom{(9.8663,10.9)
 .. controls (10.52,10.65) and (10.1175,9.55) .. (10.75,9.45)
 .. controls (11.095,9.38) and (10.98,9.1) .. (11.4643,8.9)}
\def\pwbottoma{(9.8663,12.0)
 .. controls (10.4625,11.9) and (10.06,11.0) .. (10.75,10.9)
 .. controls (11.095,10.85) and (10.98,10.3) .. (11.4643,10.1)}

\def\pw{(9.8663,8.1) .. controls (10.4655,8.35) and (10.865,7.75) .. (11.4643,7.6)}
\def\pwa{(9.8663,6.85) .. controls (10.3989,6.45) and (10.9316,6.85) .. (11.4643,6.1)}
\def\pu{(11.5307,16.40)
 .. controls (11.7305,15.65) and (12.6627,14.35) .. (12.2632,13.15)
 .. controls (12.13,12.3) and (13.4616,12.0) .. (13.4616,10.8)}
\def\pua{(12.4629,16.40)
 .. controls (13.6613,16.15) and (13.0621,14.6) .. (13.5282,13.25)
 .. controls (14.8598,12.6) and (14.5934,12.0) .. (14.9263,10.8)}
\def\pvone{(15.5255,16.40)
 .. controls (15.0595,15.65) and (17.4563,14.0) .. (17.19,13.0)
 .. controls (16.9902,12.4) and (17.5895,11.9) .. (17.5895,10.8)}
\def\pvaone{(16.3911,16.45)
 .. controls (17.6561,16.15) and (18.1221,14.55) .. (18.5882,13.2)
 .. controls (19.1873,12.5) and (19.3205,11.9) .. (19.1873,10.8)}
\def\pvtwo{(20.9184,16.40)
 .. controls (20.4523,15.80) and (22.8493,14.2) .. (23.1821,13.0)
 .. controls (23.0489,12.4) and (23.5816,11.9) .. (23.4484,10.8)}
\def\pvatwo{(21.7839,16.50)
 .. controls (23.3152,16.15) and (24.7134,14.4) .. (24.9798,13.2)
 .. controls (25.1795,12.5) and (24.9132,11.85) .. (25.0463,10.8)}
\def\pvp{(29.5737,16.40)
 .. controls (28.6416,15.85) and (29.9732,14.0) .. (29.5737,13.0)
 .. controls (29.2407,12.3) and (29.5071,11.75) .. (29.4405,10.8)}
\def\pvap{(30.6389,16.45)
 .. controls (31.5711,16.15) and (31.0384,14.3) .. (31.3048,13.1)
 .. controls (31.5711,12.3) and (31.0384,11.75) .. (30.9052,10.8)}
\foreach \p in {\pwtop,\pwtopa,\pwbottom,\pwbottoma,\pw,\pwa,\pu,\pua,\pvone,\pvaone,\pvtwo,\pvatwo,\pvp,\pvap}{\draw[track one] \p;}
\foreach \p in {\pwtopa,\pwbottoma,\pwa,\pua,\pvaone,\pvatwo,\pvap}{\draw[aux] \p;}
\foreach \xa/\ya/\xb/\yb in {8.1/5.5/9.8663/5.55,8.1/12.5/9.8663/12.0,8.1/9.0/9.8663/6.85,
 11.6639/17.90/12.4629/16.40,
 15.0/17.90/16.3911/16.45,
 21.7/17.90/21.7839/16.50,
 30.9/17.90/30.6389/16.45}{
 \draw[aux] (\xa,\ya)--(\xb,\yb);}
% Retained routes for the two further roots; their D pieces are omitted.
\foreach \xa/\ya/\xb/\yb in {8.1/5.5/9.8663/4.65,8.1/12.5/9.8663/10.9}{
 \draw[draw=auxGray,line width=1.2pt,line cap=round] (\xa,\ya)--(\xb,\yb);}
\foreach \p in {\pwtop,\pwbottom}{
 \draw[draw=auxGray,line width=1.2pt,line cap=round] \p;}
\draw[route,branchBlue] (8.1,9.0)--(9.8663,8.1);
\draw[route,branchBlue] \pw;
\draw[route,branchBlue] (11.6639,17.90)--(11.5307,16.40);
\draw[route,branchBlue] \pu;
\foreach \xa/\ya/\xb/\yb in {15.0/17.90/15.5255/16.40,
 21.7/17.90/20.9184/16.40,
 30.9/17.90/29.5737/16.40}{
 \draw[route,branchMagenta] (\xa,\ya)--(\xb,\yb);}
\foreach \p in {\pvone,\pvtwo,\pvp}{\draw[route,branchMagenta] \p;}

% P_2 runs from D to the right carrier; three retained and two unselected representatives are shown.
% With their H_{r+1} ends removed, these paths contribute to T_{r+1}^2.
\def\phfirst{(34.3673,3.95)
 .. controls (33.8348,3.68) and (33.5018,4.20) .. (32.9027,4.05)
 .. controls (32.3034,3.78) and (31.9039,3.90) .. (31.4379,4.05)}
\def\phzero{(34.3673,5) .. controls (33.8348,4.65) and (33.5018,5.35) .. (32.9027,5.1)
 .. controls (32.3034,4.65) and (31.9039,4.85) .. (31.4379,5)}
\def\ph{(34.3673,6.8) .. controls (33.8348,6.45) and (33.5018,7.15) .. (32.9027,6.9)
 .. controls (32.3034,6.45) and (31.9039,6.65) .. (31.4379,6.8)}
\def\pha{(34.3673,8.7) .. controls (33.7682,8.35) and (33.3687,9.1) .. (32.7695,8.8)
 .. controls (32.3034,8.45) and (31.8373,8.5) .. (31.4379,8.6)}
\def\phb{(34.3673,10.45) .. controls (33.8348,10.1) and (33.3687,10.8) .. (32.7695,10.5)
 .. controls (32.3034,10.15) and (31.8373,10.35) .. (31.4379,10.3)}
\foreach \p in {\phfirst,\phzero,\ph,\pha,\phb}{\draw[track two] \p;}
% The gray representative ends at w_{r+1}^1; its connection inside D is omitted.
\draw[draw=auxGray,line width=1.2pt,line cap=round] (36.15,3.95)--(34.3673,3.95);
\draw[draw=auxGray,line width=1.2pt,line cap=round] \phfirst;
\draw[route,branchBlue] (36.15,5.22)--(34.3673,5);
\draw[route,branchBlue] \phzero;
\draw[route,branchMagenta] (36.15,6.8)--(34.3673,6.8);
\draw[route,branchMagenta] \ph;
\draw[aux] (36.15,8.7)--(34.3673,8.7);
\draw[aux] \pha;
\draw[aux] (36.15,10.45)--(34.3673,10.45);
\draw[aux] \phb;
\node at (32.9027,5.65) {{\scriptsize $\vdots$}};
\node at (32.9027,9.4) {{\scriptsize $\vdots$}};

% Vertices are drawn last; line crossings alone do not denote incidence.
\foreach \xx/\yy in {9.8663/4.65,9.8663/5.55,9.8663/10.9,9.8663/12.0,
 11.4643/4.6,11.4643/5.3,11.4643/8.9,11.4643/10.1,
 8.1/5.5,
 8.1/9.0,
 8.1/12.5,
 15.4589/6.8,
 24.1143/9.25,
 11.4643/7.6,
 11.4643/6.1,
 9.8663/8.1,
 9.8663/6.85,
 11.6639/17.90,
 13.55/17.90,15.0/17.90,16.45/17.90,
 20.25/17.90,21.7/17.90,23.15/17.90,24.9132/17.90,
 29.45/17.90,30.9/17.90,32.35/17.90,33.7016/17.90,
 11.5307/16.40,
 12.4629/16.40,
 12.2632/13.15,
 13.5282/13.25,
 13.4616/10.8,
 14.9263/10.8,
 15.5255/16.40,
 16.3911/16.45,
 17.5895/10.8,
 19.1873/10.8,
 20.9184/16.40,
 21.7839/16.50,
 23.4484/10.8,
 25.0463/10.8,
 29.5737/16.40,
 30.6389/16.45,
 29.4405/10.8,
 30.9052/10.8,
 31.4379/4.05,
 34.3673/3.95,
 31.4379/5,
 34.3673/5,
 36.15/3.95,
 36.15/5.22,
 31.4379/6.8,
 31.4379/8.6,
 31.4379/10.3,
 34.3673/6.8,
 34.3673/8.7,
 34.3673/10.45,
 36.15/6.8,
 36.15/8.7,
 36.15/10.45,
 17.19/13.0,
 5.9/16.7,
 5.9/17.5,
 5.9/19.35,
 38.3/14.2,
 38.3/15.0,
 38.3/18.8}{
 \node[vertex] at (\xx,\yy) {};}
\node[align=center] at (20.9184,12.85) {$\mathcal P_1$\\[-1pt]{\scriptsize double fan}};
\node at (33.4352,3.10) {$\mathcal P_2$};
\node at (28.15,5.75) {$B_{r+1}^i$};
\node[tag] at (27.0,14.2) {$B_{r+1}^{rq+j}$};

\end{scope}

% The old U set and old T families are intentionally not drawn.
% Yellow marks displayed subsets of X; pale green marks R, itself a subset of X.
% The same pale green unites W_r and every displayed R_i as the double-fan roots.
% Pale purple marks U_{r+1} minus X.
% The two bands begin after the deleted root/terminal, so they show precisely
% the newly added paths in T_{r+1}^1 and T_{r+1}^2, rather than the full P families.
% A single legend row remains within the drawing's natural horizontal bounds.
\draw[part,fill=Xshade] (2.45,23.0) rectangle (3.65,23.6);
\node[anchor=west] at (4.05,23.3) {parts of $X$};
\draw[part,fill=Rshade,fill opacity=.25] (9.75,23.0) rectangle (10.95,23.6);
\node[anchor=west] at (11.35,23.3) {$R\subseteq X$};
\draw[part,fill=newVertexFill,fill opacity=.30] (16.05,23.0) rectangle (17.25,23.6);
\node[anchor=west] at (17.65,23.3) {$U_{r+1}\setminus X$};
\draw[track one] (22.3,23.3) .. controls (22.65,23.0) and (23.1,23.6) .. (23.5,23.3);
\node[anchor=west] at (23.9,23.3) {new paths in $\mathcal T_{r+1}^1$};
\draw[track two] (31.25,23.3) .. controls (31.6,23.0) and (32.05,23.6) .. (32.45,23.3);
\node[anchor=west] at (32.85,23.3) {new paths in $\mathcal T_{r+1}^2$};
\end{tikzpicture}%
}
\endgroup
\caption{Illustration of the separation construction.}
\label{fig:separation-step}
\end{figure}

\medskip
\noindent\textbf{Step 4.} Extending $\calM_r$ to $\calM_{r+1}$.\par\nobreak
\medskip\nobreak
We first construct a $K_{2q}$ model in each $J_i$ rooted at $R_i$.
Since $q\le t$, we have $\ln(2q)\le\ln(2t)\le p^2$.
Hence
\[
 2q\sqrt{\ln(2q)}+|R|+(r+1)q\le2pq+8t\le12t.
\]
By Equation~\eqref{eq:separation-parameters}, $\kappa(J_i)\ge N_0t>12\cst{thm:routing}t$ for every $i\in[p]$.
By Lemma~\ref{thm:routing}, each $J_i$ is $(2q,|R|+(r+1)q)$-woven.
Starting with $\calQ_0$, apply Lemma~\ref{lem:reroute} in $G'$ successively with subgraphs $J_1,\ldots,J_p$.
At step $i$, use root set $R_i$ and input linkage $\calQ_{i-1}$ to obtain a $K_{2q}$ model $\mathcal N_i$ in $J_i$ rooted at $R_i$ and an $\mathcal S$-linkage $\calQ_i$ such that
\[
 V(\calQ_i)\cap V(\mathcal N_i)=R_i,\qquad V(\calQ_i)\subseteq V(\calQ_{i-1})\cup V(J_i).
\]
For each $i\in[p]$, later applications add path vertices only in subgraphs $J_s$ with $s>i$, all disjoint from $\mathcal N_i$.
Since every vertex of $R_i$ remains a path end, $V(\calQ_p)\cap V(\mathcal N_i)=R_i$.

By the inclusion above,
\[
 V(\calQ_p)\subseteq V(\calQ_0)\cup\bigcup_{i=1}^pV(J_i)\subseteq V(\calQ_0)\cup X.
\]
Each $J_i$ is disjoint from $H_{r+1}$, $D$ and $\calM_r$.
Thus no added path vertex lies in $H_{r+1}$, $D$ or $\calM_r$.
Since $\calQ_p$ is an $\mathcal S$-linkage, it meets $H_{r+1}$ only at $W_{r+1}$ and $\calM_r$ only at $W_r$, as does $\calQ_0$ in Step 3.
No path in $\calQ_p$ has an internal vertex in $D$.
For each $z\in R\cup W_{r+1}$, let $Q_z$ be the unique path in $\calQ_p$ with end $z$.

For each $i\in[p]$, label $q$ vertices of $R_i$ as $v_i^1,\ldots,v_i^q$.
From the $pq\ge rq$ vertices of $\bigcup_{i=1}^p(R_i\setminus\{v_i^1,\ldots,v_i^q\})$, choose distinct vertices $u_1,\ldots,u_{rq}$.
For $i\in[rq]$ and $j\in[q]$, let
\[
 F_i=\{w_r^i,u_i,w_{r+1}^i\},\qquad
 F_{rq+j}=\{v_h^j:h\in[p]\}\cup\{w_{r+1}^{rq+j}\}.
\]
The sets $F_1,\ldots,F_{(r+1)q}$ are pairwise disjoint and contain $3rq+(p+1)q\le4pq\le8t$ vertices in total.
Their corresponding ends in $D$ are distinct, since the paths in $\calQ_p$ are disjoint.
By Equation~\eqref{eq:separation-parameters} and $\cst{thm:routing}\ge\cst{thm:knit}$, we have $\kappa(D)\ge N_0t>8\cst{thm:knit}t$.
By Lemma~\ref{thm:knit}, there are pairwise disjoint connected subgraphs $D_1,\ldots,D_{(r+1)q}$ of $D$ such that, for each $i\in[(r+1)q]$, $D_i$ contains the ends in $D$ of the paths $Q_z$ with $z\in F_i$ (see Fig.~\ref{fig:separation-step} for an illustration).

For $i\in[rq]$, form $B_{r+1}^i$ from $B_r^i$ by adding $D_i$, the paths $Q_z$ for $z\in F_i$, and the branch set of $\mathcal N_s$ containing $u_i$, where $s$ is the unique index with $u_i\in R_s$.
For $j\in[q]$, form $B_{r+1}^{rq+j}$ from $D_{rq+j}$ by adding the paths $Q_z$ for $z\in F_{rq+j}$ and, for each $i\in[p]$, the branch set of $\mathcal N_i$ containing $v_i^j$.
The paths in $\calQ_p$ are disjoint, meet each $\mathcal N_i$ only at $R_i$, and have no internal vertex in $D$.
The subgraphs $D_1,\ldots,D_{(r+1)q}$ are disjoint.
Hence the branch sets of $\calM_{r+1}$ are connected and pairwise disjoint.

Adjacencies between the enlarged old branch sets are inherited from $\calM_r$.
For $i\in[rq]$, let $s$ be the index with $u_i\in R_s$.
In $\mathcal N_s$, the branch set containing $u_i$ is adjacent to the branch set containing $v_s^j$ for every $j\in[q]$.
Thus $B_{r+1}^i$ is adjacent to every new branch set.
Any two new branch sets are adjacent within $\mathcal N_1$.
Therefore $\calM_{r+1}=(B_{r+1}^1,\ldots,B_{r+1}^{(r+1)q})$ is a $K_{(r+1)q}$ model.
Since $\calQ_p$ meets $H_{r+1}$ only at $W_{r+1}$, each $B_{r+1}^i$ meets $H_{r+1}$ only at $w_{r+1}^i$.
Thus~\ref{inv:tangent} holds.

\medskip
\noindent\textbf{Step 5.} Constructing $\calT_{r+1}^1$ and $\calT_{r+1}^2$.\par\nobreak
\medskip\nobreak
Recall that $\calP_1$ is the double fan from $R$ to $V(D)$ in $G'$
defined in Step 1, and $\calP_2$ is the family of $2(r+1)q$ disjoint
$(V(H_{r+1}),V(D))$-paths in $G'\setminus R$ defined in Step 3.
Let
\[
 \calT_{r+1}^1=\calT_r^1\cup\{P\setminus R:P\in\calP_1\},\quad
 \calT_{r+1}^2=\calT_r^2\cup\{P\setminus V(H_{r+1}):P\in\calP_2\}.
\]
The paths in $\calP_1\cup\calP_2$ lie in $G'$ and avoid $V(\calT_r^1\cup\calT_r^2)$.
Thus the paths in each of $\calT_{r+1}^1,\calT_{r+1}^2$ are pairwise disjoint.
By the construction in Step 2, $H_{r+1}$ is disjoint from
$V(\calT_r^1\cup\calT_r^2\cup\calP_1)$.
Each path in $\calP_2$ meets $H_{r+1}$ only at the end deleted when defining $\calT_{r+1}^2$.
Hence both $\calT_{r+1}^1$ and $\calT_{r+1}^2$ avoid $H_{r+1}$.
By~\ref{inv:paths} at stage $r$, $|\calT_{r+1}^1|+|\calT_{r+1}^2|\le16tr+12t+4t=16t(r+1)$.
This proves~\ref{inv:paths}.

\medskip
\noindent\textbf{Step 6.} Constructing $U_{r+1}$.\par\nobreak
\medskip\nobreak
For each $i\in[2]$, let $Z_i$ be the set of vertices outside
$X\cup V(\calT_{r+1}^1\cup\calT_{r+1}^2)$ that have a neighbor on at least $2t$
distinct paths in $\calT_{r+1}^i$.
Let
\[
 U_{r+1}=X\cup Z_1\cup Z_2\cup\bigl(V(\calP_2)\cap V(H_{r+1})\bigr).
\]
We next verify that~\ref{inv:size}--\ref{inv:neighbors} hold.

For each $i\in[2]$, consider the bipartite minor of $G$ with one part
formed by contracting each path in $\calT_{r+1}^i$ to a single vertex
and the other part $Z_i$.
Every vertex of $Z_i$ has degree at least $2t$ in this minor.
By Lemma~\ref{lem:bipartite}, $|Z_i|\le\cst{lem:bipartite}(\ln t)|\calT_{r+1}^i|$.
Thus
\[
  |U_{r+1}\setminus X|\le2(r+1)q+|Z_1|+|Z_2|\le4t+16\cst{lem:bipartite}t(\ln t)(r+1).
\]
By Inequality~\eqref{eq:separation-small-pieces}, Equation~\eqref{eq:separation-parameters},
\ref{inv:size} at stage $r$, and the inequalities $p+1\le3\ln t$ and $r+1\le p\le2\ln t$,
\begin{align*}
 |U_{r+1}|\le {}& |U_r|+(p+1)\cst{thm:small-connected}^2t(\ln t)^3+4t
              +16\cst{lem:bipartite}t(\ln t)(r+1)\\
 \le {}& t+rN_0t(\ln t)^4+(3\cst{thm:small-connected}^2+4+32\cst{lem:bipartite})t(\ln t)^4
 \le t+(r+1)N_0t(\ln t)^4.
\end{align*}
Recall that, in the construction of $\calM_{r+1}$ in Step 4, we use the old model
$\calM_r$, the subgraphs $D_1,\ldots,D_{(r+1)q}$ of $D$, branch sets of
$\mathcal N_i$ in $J_i$ for $i\in[p]$, and paths in $\calQ_p$.
Here $V(D)\cup\bigcup_{i=1}^pV(J_i)\subseteq X$ and
$V(\calQ_p)\subseteq V(\calQ_0)\cup X\subseteq V(\calP_1\cup\calP_2)\cup X$.
The ends removed from $\calP_1$ and $\calP_2$ in Step 5 all lie in $U_{r+1}$.
Thus, by~\ref{inv:size} at stage $r$ and the definitions of $\calT_{r+1}^1,\calT_{r+1}^2$, we have
\[
 V(\calM_{r+1})\subseteq V(\calM_r)\cup X\cup V(\calP_1\cup\calP_2)
 \subseteq U_{r+1}\cup V(\calT_{r+1}^1\cup\calT_{r+1}^2).
\]
Also, $W_{r+1}\subseteq V(\calP_2)\cap V(H_{r+1})\subseteq U_{r+1}$, so~\ref{inv:size} holds.

Outside $U_{r+1}$, the new path families consist of the vertices in
$Y\setminus U_{r+1}$ and those of $\calP_2$ outside $U_{r+1}\cup Y$.
All ends of $\calP_2$ lie in $U_{r+1}$, and $|\calP_2|\le4t$.
By Lemma~\ref{lem:short-fan}(2) in $G'\setminus R$, the graph induced
by $V(\calP_2)\setminus(U_{r+1}\cup Y)$ is $(8t-1)$-degenerate.
Every vertex of $V(\calP_2)\setminus(U_{r+1}\cup Y)$ lies outside $X\cup Y$
and hence has at most $48t$ neighbors in $Y$ by Step 2.
Delete the new vertices first, each with at most $56t-1$ remaining neighbors.
The remaining graph is $G[Y\setminus U_{r+1}]$, which is
$100t$-degenerate by Step 2. This proves~\ref{inv:degenerate}.

Let $v$ be any vertex outside $U_{r+1}\cup V(\calT_{r+1}^1\cup\calT_{r+1}^2)$.
Then $v\in V(G')\setminus R$ and $v$ lies on no path in $\calP_1\cup\calP_2$.
By Lemma~\ref{lem:short-fan}(1), Remark~\ref{rem:short-double-fan}, and~\ref{inv:neighbors}
at stage $r$, $v$ has at most three neighbors on each path in $\calT_{r+1}^1\cup\calT_{r+1}^2$.
Since $Z_1\cup Z_2\subseteq U_{r+1}$, we have $v\notin Z_1\cup Z_2$.
By the definitions of $Z_1,Z_2$, the vertex $v$ has neighbors on at most
$2t-1$ paths in $\calT_{r+1}^1$ and at most $2t-1$ paths in $\calT_{r+1}^2$.
Thus $v$ has at most $3(2t-1)+3(2t-1)<12t$ neighbors in
$V(\calT_{r+1}^1\cup\calT_{r+1}^2)$, so~\ref{inv:neighbors} holds.

This completes the construction of the tuple at stage $r+1$
satisfying~\ref{inv:carrier}--\ref{inv:neighbors}.

\smallskip
Repeating the six steps for $r=0,\ldots,p-1$, we obtain a $K_{pq}$ model $\calM_p$.
Since $pq\ge t$, $G$ contains a $K_t$ minor, contrary to the hypothesis.
Hence $L$ is $k$-separable in $G$.
\end{proof}

\section{Proof of Theorem~\ref{thm:main}}\label{sec:main-proof}

The separation lemma provides three disjoint subgraphs to which induction
applies. We join the resulting models into a larger rooted complete minor.
The required paths are chosen first, and the rooted model is then
constructed in a subgraph that meets these paths only at the chosen roots.

The recursion below follows the construction in the proof of
Theorem~7.1 in~\cite{DelcourtPostle2025}.
We state and prove the list version with the bounds needed here.

Let $a_0$ be the least positive integer such that
$14\cdot2^{a_0}\ge \tst{thm:local-list}$.

\begin{lemma}\label{thm:recursive}
There are absolute positive integers $N,C=\cst{thm:recursive}$ such that,
for all integers $a\ge a_0$ and $0\le b\le a$, every
$N2^{a-b}3^b$-connected graph $G$ with a bad list assignment $L$ satisfying
$|L|\ge C\bigl(3^a+2^{a-b}3^b\bigr)$
is $(2^{a-b}3^b,\,3\cdot2^{a-b}3^b)$-woven.
\end{lemma}

\begin{proof}[Proof of Theorem~\ref{thm:main} assuming Lemma~\ref{thm:recursive}]
Let $t\ge3^{a_0}$ be an integer, and choose an integer $a$ such that
$t\le3^a<3t$. Then $a\ge a_0$.
We take $C=3(4N+3\cst{thm:recursive})$.
Suppose that a $K_t$-minor-free graph $G$ has a bad list assignment $L$
with $|L|=Ct$. Since $3^a<3t$,
\[
 |L|-4N3^a\ge Ct-4N3^a>3\cst{thm:recursive}3^a.
\]
In particular, $|L|\ge4N3^a+1$.
By Lemma~\ref{thm:bad-connected}, $G$ has an induced
$N3^a$-connected subgraph $H$ with a bad sublist $L_H\subseteq L$
satisfying $|L_H|\ge|L|-4N3^a\ge2\cst{thm:recursive}3^a$.
By Lemma~\ref{thm:recursive} with $b=a$, $H$ is $(3^a,3^{a+1})$-woven.
Thus $H$ has a $K_{3^a}$ minor and hence a $K_t$ minor, a contradiction.

For $2\le t<3^{a_0}$, Theorem~\ref{thm:density} below and greedy list coloring imply
\[
 \ch(G)\le2\cst{thm:density}t\sqrt{\ln t}+1
 \le\bigl(2\cst{thm:density}\sqrt{a_0\ln3}+1\bigr)t.
\]
Enlarging $C$ to a sufficiently large integer if necessary, every
$K_t$-minor-free graph is
$Ct$-choosable for every integer $t\ge2$.
\end{proof}

It remains to prove Lemma~\ref{thm:recursive}.
We first record a density bound and a preparatory lemma, then give the inductive proof.

\subsection{Preliminaries}\label{subsec:preparation}

Kostochka~\cite{Kostochka1984} and
Thomason~\cite{Thomason1984,Thomason2001} proved the following density bound.

\begin{theorem}[\cite{Kostochka1984,Thomason1984,Thomason2001}]
\label{thm:density}
There is an absolute constant $C=\cst{thm:density}\ge1$ such that,
for every integer $t\ge2$, every nonempty graph $G$ with
$\mathsf{d}(G)\ge Ct\sqrt{\ln t}$
contains a $K_t$ minor.
\end{theorem}

The next lemma is based on the construction of a small connected
subgraph and a double fan in the proof of Theorem~7.1
in~\cite{DelcourtPostle2025}.

\begin{lemma}\label{lem:prepare}
Let $N,p$ be positive integers with $14p\ge \tst{thm:local-list}$ and
$N>100\cst{lem:deletion}\cst{thm:small-connected}\cst{thm:routing}$.
Let $G$ be a
$(N-7)p$-connected $K_{14p}$-minor-free graph, and let
$Z\subseteq V(G)$ have size $7p$. Suppose that $L$ is a bad list
assignment on $G$ with $|L|\ge6Np$.
Then there are disjoint induced subgraphs $D,H$ of $G\setminus Z$
satisfying the following conditions.
\begin{enumerate}
\item $\kappa(D)\ge14\cst{thm:routing}p$ and $\kappa(H)\ge Np$.
\item There is a double fan $\calP$ in $G$ from $Z$ to $V(D)$ avoiding $H$.
\item $H$ has a bad sublist $L_H\subseteq L$ with $|L_H|\ge|L|-6Np$.
\end{enumerate}
\end{lemma}

\begin{proof}
Since $\delta(G)\ge(N-7)p$, we have
\[
 \mathsf d(G\setminus Z)\ge\frac{(N-14)p}{2}
 \ge\cst{thm:small-connected}(14\cst{thm:routing}p).
\]
By Lemma~\ref{thm:small-connected}, with $t=14p$ and
$k=14\cst{thm:routing}p$, the graph $G\setminus Z$ contains an
induced $14\cst{thm:routing}p$-connected subgraph $D$ with
$\mathsf v(D)\le\cst{thm:small-connected}^2(14p)(\ln(14p))^3$.
Since $G$ is at least $14p$-connected and $\mathsf v(D)\ge14p=2|Z|$,
by Lemma~\ref{lem:double-fan}, there is a double fan $\calP$ from
$Z$ to $V(D)$. Choose one of minimum total length in $G$.

Let $X=Z\cup V(D)$ and $Y=V(\calP)\setminus X$.
By Remark~\ref{rem:short-double-fan}, with $\ell=14p$, the graph
$G[Y]$ is $42p$-degenerate, and every vertex outside $X\cup Y$
has at most $42p$ neighbors in $Y$.
The choice of $\tst{thm:local-list}$ before Lemma~\ref{cor:many-pieces}, together
with $14p\ge \tst{thm:local-list}$, ensures that
\[
 |X|\le7p+\cst{thm:small-connected}^2(14p)(\ln(14p))^3
 \le(14p)(\ln(14p))^5.
\]
Also $\cst{lem:deletion}(56p+1)\le57\cst{lem:deletion}p<Np$.
By Lemma~\ref{lem:deletion}, with $t=14p$ and $d=42p$,
$G\setminus(X\cup Y)$ has a bad sublist of size at least $|L|-Np$.
Since $|L|-Np\ge5Np\ge4Np+1$, by Lemma~\ref{thm:bad-connected},
it contains an induced $Np$-connected subgraph $H$ with a bad sublist
$L_H\subseteq L$ satisfying $|L_H|\ge|L|-5Np\ge|L|-6Np$.
The graph $H$ is disjoint from $D$ and every path in $\calP$.
\end{proof}

\subsection{Proof of Lemma~\ref{thm:recursive}}\label{subsec:recursive-proof}

We take
\begin{equation}\label{eq:choose-recursive}
 N=\ceil{100\cst{lem:deletion}\cst{thm:small-connected}\cst{thm:routing}}+1,\quad
 C=\ceil{\max\{60\cst{thm:density},3(10N+42\cst{lem:separation}+17)\}}.
\end{equation}
Let $a\ge a_0$ be fixed. We prove the lemma by induction on $b=0,1,\ldots,a$.
Let $p=2^{a-b}3^b$, and let $G$ be an $Np$-connected graph with a bad list assignment $L$ satisfying
\[
 |L|\ge C\left(3^a+p\right).
\]
We shall prove that $G$ is $(p,3p)$-woven. Let $R=\{r_1,\ldots,r_p\}$, $S=\{s_j:j\in[3p]\}$ and $T=\{t_j:j\in[3p]\}$ be arbitrary subsets of $V(G)$ with $|R|=p$, $|S|=|T|=3p$ and $s_j\ne t_k$ whenever $j\ne k$. We will find an $R$-rooted $K_p$ model and a $\{(s_j,t_j)\}_{j\in[3p]}$-linkage whose common vertices with the model are exactly $R\cap(S\cup T)$.

\begin{claim}\label{clm:disjoint-roots}
It suffices to consider the case where $R\cap(S\cup T)=\varnothing$.
\end{claim}

We use the root-separation argument from the proof of Theorem~7.1
of Delcourt and Postle~\cite{DelcourtPostle2025}.

\begin{proof}
For each $r_i\in R\cap(S\cup T)$, choose a distinct neighbor $r'_i$
outside $R\cup S\cup T$. This is possible since
$|R\cup S\cup T|+p-1\le8p-1<\delta(G)$.
Let $r'_i=r_i$ for $r_i\notin S\cup T$, and let $R'=\{r'_1,\ldots,r'_p\}$.
Then $R'\cap(S\cup T)=\varnothing$.
Suppose that the required conclusion holds for $R'$.
Let $\calM$ be the resulting $R'$-rooted $K_p$ model and $\calQ$ the linkage
disjoint from $\calM$. For each $i$ with $r'_i\ne r_i$, add $r_i$ and
the edge $r_ir'_i$ to the branch set containing $r'_i$.
All added vertices belong to $S\cup T$ and hence were avoided by $\calM$.
The resulting branch sets therefore form an $R$-rooted $K_p$ model
whose intersection with $\calQ$ is exactly $R\cap(S\cup T)$.
\end{proof}

By Claim~\ref{clm:disjoint-roots}, we may assume that $R\cap(S\cup T)=\varnothing$. Choose a set $Z\supseteq R\cup S\cup T$ with $|Z|=7p$.

Suppose that $G$ contains a $K_{14p}$ minor. Since $\kappa(G)\ge Np\ge7p$, by Lemma~\ref{lem:rooted-minor}, $G$ contains a $K_{7p}$ model rooted at $Z$. Retain the branch sets rooted at $R$. For each $j\in[3p]$ with $s_j\ne t_j$, choose an $(s_j,t_j)$-path in the union of the branch sets rooted at $s_j,t_j$ and an edge joining them. If $s_j=t_j$, use the trivial path consisting only of $s_j$. These paths form a linkage disjoint from the retained model, as required. We may therefore assume that $G$ is $K_{14p}$-minor-free.

\medskip
\noindent\textbf{Base case: $b=0$.}
\medskip

Here $p=2^a$. By Observation~\ref{obs:bad-density}, the assumption on $|L|$, and the choice of $C$ in Equation~\eqref{eq:choose-recursive}, $G$ has an induced subgraph $H$ such that
\[
 \mathsf{d}(H)
 \ge\frac{|L|}{2}
 \ge\frac{C3^a}{2}
 \ge30\cst{thm:density}3^a
 >14\cst{thm:density}2^a\sqrt{a\ln28}
 \ge14\cst{thm:density}2^a\sqrt{\ln(14\cdot2^a)}.
\]
Here we used $2^a\sqrt{a}\le(2\sqrt2)^a<3^a$, $\sqrt{\ln28}<2$, and $\ln(14\cdot2^a)\le a\ln28$. So by Theorem~\ref{thm:density} with $t=14\cdot2^a$, $H$ (and hence $G$) contains a $K_{14p}$ minor, a contradiction. This proves the base case.

\Needspace{4\baselineskip}
\medskip
\noindent\textbf{Induction step: $1\le b\le a$.}
\medskip

Assume the lemma holds at $b-1$: every graph $J$ with
$\kappa(J)\ge2Np/3$ and a bad list assignment $L_J$ satisfying
$|L_J|\ge C(3^a+2p/3)$ is $(2p/3,2p)$-woven.

By the choice of $a_0$, $14p\ge14\cdot2^a\ge \tst{thm:local-list}$.
Also, $C\ge6N$ by Equation~\eqref{eq:choose-recursive},
so $|L|\ge Cp\ge6Np$.
By Lemma~\ref{lem:prepare}, there are disjoint induced subgraphs $D,H$ of $G\setminus Z$ and a double fan $\mathcal P_1$ from $Z$ to $V(D)$ avoiding $H$, where
\begin{equation}\label{eq:induction-connectivity}
 \kappa(D)\ge14\cst{thm:routing}p,\qquad \kappa(H)\ge Np.
\end{equation}
Moreover, $H$ has a bad sublist $L_H\subseteq L$ satisfying
\begin{equation}\label{eq:induction-H-list}
 |L_H|\ge|L|-6Np.
\end{equation}

We use $D$ to connect $R$ to $H$ and to link the pairs $(s_j,t_j)$.

By the choice of $N$ in Equation~\eqref{eq:choose-recursive}, $\kappa(G\setminus Z)\ge\kappa(G)-|Z|\ge(N-7)p\ge2p$. By Inequalities~\eqref{eq:induction-connectivity}, both $D$ and $H$ have at least $2p$ vertices. By Lemma~\ref{lem:menger}, there is a family $\mathcal P_2$ of $2p$ disjoint $(V(H),V(D))$-paths in $G\setminus Z$. Recall that $\mathcal P_1$ is a double fan from $Z$ to $V(D)$ avoiding $H$. The sets $Z,V(H),V(D)$ are pairwise disjoint. By Lemma~\ref{lem:join-fans} (applied to the sets $Z,V(H),V(D)$, with $r=p$), the subgraph formed by $\mathcal P_1$ and $\mathcal P_2$ contains a family $\mathcal R_1$ of $p$ $(R,V(D))$-paths, a family $\calQ_1$ of $6p$ $(Z\setminus R,V(D))$-paths, and a family $\mathcal R_2$ of $p$ $(V(H),V(D))$-paths. All $8p$ paths in $\mathcal R_1\cup\calQ_1\cup\mathcal R_2$ are pairwise disjoint and have no internal vertex in $Z\cup V(H)\cup V(D)$. Each vertex of $Z$ is an end of one path in $\mathcal R_1\cup\calQ_1$. Let $R_H=\{r_1^*,\ldots,r_p^*\}$ be the set of ends in $H$ of the paths in $\mathcal R_2$ (see Fig.~\ref{fig:induction-step} for an illustration).

Recall that $\kappa(D)\ge14\cst{thm:routing}p\ge10(4p)$ by Inequalities~\eqref{eq:induction-connectivity} and $\cst{thm:routing}\ge10$. By Lemma~\ref{thm:linked} (with $k=4p$), $D$ is $4p$-linked. For $i\in[p]$, pair the end in $D$ of the path in $\mathcal R_1$ from $r_i$ with the end in $D$ of the path in $\mathcal R_2$ from $r_i^*$. For $j\in[3p]$, pair the ends in $D$ of the paths in $\calQ_1$ from $s_j$ and $t_j$; these are the same path and the same end when $s_j=t_j$. Terminals belonging to different pairs are distinct. By the $4p$-linked property, $D$ contains a family $\mathcal R_3$ of $p$ paths joining the first $p$ pairs and a family $\calQ_2$ of $3p$ paths joining the other $3p$ pairs, with all $4p$ paths pairwise disjoint.

Extend each path in $\mathcal R_3$ along the corresponding paths in $\mathcal R_1$ and $\mathcal R_2$, and let $\mathcal R$ be the resulting $\{(r_i,r_i^*)\}_{i\in[p]}$-linkage. For each $j\in[3p]$ with $s_j\ne t_j$, extend the corresponding path in $\calQ_2$ along the two corresponding paths in $\calQ_1$, and call the resulting path $Q_j$. If $s_j=t_j$, let $Q_j$ be the trivial path consisting only of $s_j$. Let $\calQ=\{Q_j:j\in[3p]\}$. The paths of $\calQ_1$ starting in $Z\setminus(R\cup S\cup T)$ are not used. Since the connecting paths have no internal vertex in $Z$, all $4p$ paths in $\mathcal R\cup\calQ$ are pairwise disjoint. Moreover, $V(\calQ)\cap V(H)=\emptyset$ and $V(\mathcal R)\cap V(H)=R_H$.

\begin{figure}[!htb]
\centering
\begingroup
% Inline TikZ for the induction-step construction.
\definecolor{branchBlue}{HTML}{1479F5}
\definecolor{branchRose}{HTML}{E72C9B}
\definecolor{branchGreen}{HTML}{11B45B}
\definecolor{branchBrick}{HTML}{B64035}
\definecolor{branchYellow}{HTML}{E6B800}
\definecolor{branchPurple}{HTML}{8B45E6}
\definecolor{pathred}{HTML}{D8263D}
\definecolor{icGray}{HTML}{7A8793}
\definecolor{icBorder}{HTML}{B6C0C9}
\resizebox{\linewidth}{!}{%
\begin{tikzpicture}[x=.6cm,y=-.6cm,
  every node/.style={font=\fontsize{9}{10.5}\selectfont,text=black,
    inner sep=0pt,outer sep=0pt},
  icRegion/.style={draw=icBorder,line width=.5pt},
  icPart/.style={draw=icGray!70,line width=.4pt},
  icPath/.style={line width=1.4pt,line cap=round,line join=round},
  icBridge/.style={preaction={draw=white,line width=3.4pt}},
  icEdge/.style={draw=icGray,line width=.5pt},
  icVertex/.style={circle,fill=white,draw=black,line width=.45pt,inner sep=0pt,minimum size=3.8pt},
  icNeighbor/.style={circle,fill=white,draw=black,line width=.45pt,inner sep=0pt,minimum size=3.8pt},
  icBlock/.style={ellipse,draw=none,
    minimum width=8mm,minimum height=5.5mm,inner sep=0pt,
    font=\fontsize{8.5}{10}\selectfont,text=black},
  icLabel/.style={fill=white,inner sep=1pt},
  icNote/.style={font=\fontsize{8.2}{9.6}\selectfont,text=black}]
\path[use as bounding box] (3.0,0) rectangle (27.75,20.9);

% The construction lies in G; D and H avoid the protected set Z.
\draw[icRegion,dashed] (3.3,.45) rectangle (27.45,20.15);
\node[icLabel,anchor=west] at (7.35,.45) {$G$};
\draw[icRegion,fill=icGray!3] (3.45,.70) rectangle (5.9,19.75);
\node[icLabel] at (3.87,.70) {$Z$};
\draw[icRegion,fill=icGray!5] (7.35,1.35) rectangle (10.3,19.75);
\node[icLabel] at (8.825,1.35) {$D$};
\node[icNote] at (8.825,19.15) {$4p$-linked};
\draw[icRegion,fill=icGray!3]
  (12.3,1.35) rectangle (27.0,19.75);
\node at (12.95,1.88) {$H$};

% Only representative roots and noncoincident terminals in Z are shown.
\foreach \xleft/\xright in {3.67/5.68}{
  \draw[icPart,fill=white] (\xleft,1.30) rectangle (\xright,3.80);
  \draw[icPart,fill=white] (\xleft,4.20) rectangle (\xright,6.70);
  \draw[icPart,fill=white] (\xleft,7.20) rectangle (\xright,19.20);
}
\node[icLabel] at (4.22,1.30) {$S$};
\node[icLabel] at (4.22,4.20) {$T$};
\node[icLabel] at (4.22,7.20) {$R$};

% Vertical ellipses above and below the representative terminal vertices.
\foreach \xx in {4.70}{
  \foreach \yy in {1.55,3.48,4.48,6.38}{
    \node at (\xx,\yy) {\scalebox{.65}{$\vdots$}};}}

% The three disjoint bad subgraphs and their local woven models.
\foreach \top/\bottom/\innerTop/\innerBottom/\h in {
  3.2/7.85/4.0/7.55/1,
  8.9/13.55/9.7/13.25/2,
  14.6/19.25/15.4/18.95/3}{
  \draw[icPart,fill=icGray!4] (19.14,\innerTop) rectangle (23.34,\innerBottom);
  \node[icLabel] at (21.24,\innerTop) {$J_{\h}$};
}
\node[icNote] at (21.24,4.58) {$K_{2p/3}$ model};
\node[icNote] at (21.24,10.28) {$K_{2p/3}$ model};
\node[icNote] at (21.24,15.98) {$K_{2p/3}$ model};

% Neutral backgrounds distinguish the root partition from the colored paths.
\draw[draw=black!30,fill=black!3,line width=.45pt]
  (12.42,3.30) rectangle (14.18,18.40);
\node[icLabel,fill=black!3] at (13.30,3.30) {$R_H$};
\node[icNote] at (13.30,18.95) {$|R_H|=p$};
\foreach \top/\bottom/\part in {4.95/8.35/2,9.05/12.65/3,13.40/16.90/1}{
  \draw[draw=black!45,fill=black!8,line width=.5pt]
    (12.60,\top) rectangle (14.00,\bottom);
  \node[icLabel,fill=black!8] at (13.30,\top) {$A_{\part}$};
}
% The twelve shown neighbors belong to Y; other neighbors are omitted.
\draw[icPart,densely dashed,fill=white] (14.60,4.55) rectangle (15.43,17.5);
\node[icLabel] at (15.015,4.55) {$Y$};
\node[icNote] at (15.9,18.95) {$|Y|=2p$};

% Two representative nontrivial terminal paths; singleton paths and padding vertices are omitted.
% The T terminals are in reverse order so that the two paths in D are nested.
\foreach \yy in {2.0,3.05,4.90,5.95}{
  % Q_1 has the same shallow, medium-period wave as R_1.
  \draw[icPath,pathred] (4.7,\yy)
    .. controls (4.921,{\yy+.20}) and (5.142,{\yy+.20}) .. (5.3625,\yy)
    .. controls (5.583,{\yy-.20}) and (5.804,{\yy-.20}) .. (6.025,\yy)
    .. controls (6.246,{\yy+.20}) and (6.467,{\yy+.20}) .. (6.6875,\yy)
    .. controls (6.908,{\yy-.20}) and (7.129,{\yy-.20}) .. (7.35,\yy);
}
% The outer path in Q_2 represents another pair; its ends are unlabeled.
\foreach \yy in {2.0,5.95}{
  \draw[icPath,pathred] (7.35,\yy)
    .. controls (7.488,{\yy-.25}) and (7.625,{\yy-.25}) .. (7.7625,\yy)
    .. controls (7.900,{\yy+.25}) and (8.038,{\yy+.25}) .. (8.175,\yy)
    .. controls (8.313,{\yy-.25}) and (8.450,{\yy-.25}) .. (8.5875,\yy)
    .. controls (8.725,{\yy+.25}) and (8.863,{\yy+.25}) .. (9.0,\yy);
}
\draw[icPath,pathred] (9.0,2.0)
  .. controls (10.15,2.0) and (10.15,5.95) .. (9.0,5.95);
% The inner path in Q_2 joins the i-th pair and avoids the outer path.
\foreach \yy in {3.05,4.90}{
  \draw[icPath,pathred] (7.35,\yy)
    .. controls (7.445,{\yy-.20}) and (7.54,{\yy-.20}) .. (7.635,\yy)
    .. controls (7.730,{\yy+.20}) and (7.825,{\yy+.20}) .. (7.92,\yy)
    .. controls (8.015,{\yy-.20}) and (8.11,{\yy-.20}) .. (8.2,\yy);
}
\draw[icPath,pathred] (8.2,3.05)
  .. controls (9.1,3.05) and (9.1,4.90) .. (8.2,4.90);
\foreach \x in {4.7}{
  \foreach \y in {2.0,3.05,4.90,5.95}{\node[icVertex] at (\x,\y) {};}}
% Only the inner representative pair is labeled.
\foreach \x/\y/\lab in {4.7/3.05/s_i,4.7/4.90/t_i}{
  \node[anchor=west,font=\fontsize{7.7}{9}\selectfont] at (\x+.15,\y-.22) {$\lab$};
}
% Further members of both path families are omitted.
\foreach \yy in {2.525,5.425}{
  \node at (6.025,\yy) {\scalebox{.6}{$\vdots$}};}
\node at (9.37,3.98) {\scalebox{.65}{$\cdots$}};
\foreach \y in {2.0,3.05,4.90,5.95}{\node[icVertex] at (7.35,\y) {};}
\node[text=black] at (6.42,3.98) {$\mathcal Q_1$};
\node[text=black] at (7.97,3.98) {$\mathcal Q_2$};

% Six representative root paths: R_1, R_3, R_2 in this geometric order.
\foreach \i/\yy/\col in {
  1/6.2/branchBlue,2/7.55/branchRose,3/10.25/branchGreen,
  4/11.9/branchBrick,5/14.6/branchYellow,6/16.25/branchPurple}{
  \pgfmathsetmacro{\lefty}{\yy+2.0}
  % R_1: two shallow waves.
  \draw[icPath,\col] (4.7,\lefty)
    .. controls (4.921,{\lefty+.20}) and (5.142,{\lefty+.20}) .. (5.3625,\lefty)
    .. controls (5.583,{\lefty-.20}) and (5.804,{\lefty-.20}) .. (6.025,\lefty)
    .. controls (6.246,{\lefty+.20}) and (6.467,{\lefty+.20}) .. (6.6875,\lefty)
    .. controls (6.908,{\lefty-.20}) and (7.129,{\lefty-.20}) .. (7.35,\lefty);
  % R_3: six tight waves inside D, distinct from the two broad R_1 waves.
  \draw[icPath,\col] (7.35,\lefty)
    .. controls (7.4319,{\lefty-.25}) and (7.5139,{\lefty-.25}) .. (7.5958,\lefty)
    .. controls (7.6778,{\lefty+.25}) and (7.7597,{\lefty+.25}) .. (7.8417,\lefty)
    .. controls (7.9236,{\lefty-.25}) and (8.0056,{\lefty-.25}) .. (8.0875,\lefty)
    .. controls (8.1694,{\lefty+.25}) and (8.2514,{\lefty+.25}) .. (8.3333,\lefty)
    .. controls (8.4153,{\lefty-.25}) and (8.4972,{\lefty-.25}) .. (8.5792,\lefty)
    .. controls (8.6611,{\lefty+.25}) and (8.7431,{\lefty+.25}) .. (8.825,\lefty)
    .. controls (8.9069,{\lefty-.25}) and (8.9889,{\lefty-.25}) .. (9.0708,\lefty)
    .. controls (9.1528,{\lefty+.25}) and (9.2347,{\lefty+.25}) .. (9.3167,\lefty)
    .. controls (9.3986,{\lefty-.25}) and (9.4806,{\lefty-.25}) .. (9.5625,\lefty)
    .. controls (9.6444,{\lefty+.25}) and (9.7264,{\lefty+.25}) .. (9.8083,\lefty)
    .. controls (9.8903,{\lefty-.25}) and (9.9722,{\lefty-.25}) .. (10.0542,\lefty)
    .. controls (10.1361,{\lefty+.25}) and (10.2181,{\lefty+.25}) .. (10.3,\lefty);
  % R_2 descends back to the original roots in H, with the same broad wave.
  \draw[icPath,\col] (10.3,\lefty)
    .. controls (10.783,{\yy+2.0669}) and (11.267,{\yy+1.7331}) .. (11.75,{\yy+1.0})
    .. controls (12.233,{\yy+.2669}) and (12.717,{\yy-.0669}) .. (13.2,\yy);
  \foreach \xx in {4.7,7.35,10.3}{
    \node[icVertex,draw=black] at (\xx,\lefty) {};}
  \node[icVertex] at (13.2,\yy) {};
  \node[anchor=west] at (4.87,\lefty-.32) {$r_{\i}$};
  \node[fill=black!8,inner sep=.4pt] at (13.2,\yy-.56) {$r_{\i}^*$};
}
\node at (6.42,10.88) {$\mathcal R_1$};
\node at (8.825,10.88) {$\mathcal R_3$};
\node at (11.73,9.88) {$\mathcal R_2$};
\node[icNote] at (5.07,18.80) {\scalebox{.7}{$\vdots$}};
\node[icNote] at (13.2,17.65) {$\vdots$};

% Each root uses two local models. White gaps distinguish crossings from vertices.
% Selected edges end at the twelve displayed Y vertices.
\foreach \yy/\up/\down/\col in {
  6.2/5.85/6.55/branchBlue,7.55/7.2/7.9/branchRose,
  10.25/9.9/10.6/branchGreen,11.9/11.55/12.25/branchBrick,
  14.6/14.25/14.95/branchYellow,16.25/15.9/16.6/branchPurple}{
  \draw[icEdge,\col] (13.2,\yy)--(15.015,\up);
  \draw[icEdge,\col] (13.2,\yy)--(15.015,\down);
}
% B_1 and B_2 use J_1 and J_3.
\draw[icPath,branchBlue] (15.015,6.55)
  .. controls (16.7,6.25) and (17.55,5.35) .. (20.03,5.35);
\draw[icPath,branchBlue] (15.015,5.85)
  .. controls (15.9,5.0) and (15.25,2.95) .. (16.65,2.95)
  --(24.55,2.95) .. controls (25.2,2.95) and (25.2,3.15) .. (25.2,3.8)
  --(25.2,16.02) .. controls (25.2,16.75) and (24.8,16.75) .. (24.15,16.75)
  --(22.45,16.75);
\draw[icPath,branchRose] (15.015,7.9)
  .. controls (16.7,7.9) and (17.55,6.75) .. (20.03,6.75);
\draw[icPath,branchRose,icBridge] (15.015,7.2)
  --(15.35,7.2)
  .. controls (16.05,7.2) and (16.45,6.85) .. (16.45,6.15)
  --(16.45,3.30)
  .. controls (16.45,2.60) and (16.75,2.3) .. (17.45,2.3)
  --(25.75,2.3) .. controls (26.45,2.3) and (26.45,2.6) .. (26.45,3.2)
  --(26.45,17.4) .. controls (26.45,18.15) and (26.0,18.15) .. (25.3,18.15)
  --(22.45,18.15);
% B_3 and B_4 use J_1 and J_2.
\draw[icPath,branchGreen] (15.015,9.9)
  .. controls (16.3,9.9) and (16.6,8.7) .. (18.0,8.7)
  --(23.5,8.7) .. controls (24.15,8.7) and (24.15,8.45) .. (24.15,7.8)
  --(24.15,5.9) .. controls (24.15,5.35) and (23.8,5.35) .. (23.25,5.35)
  --(22.45,5.35);
\draw[icPath,branchGreen] (15.015,10.6)
  .. controls (16.4,10.6) and (17.5,11.05) .. (20.03,11.05);
\draw[icPath,branchBrick,icBridge] (15.015,11.55)
  .. controls (16.1,11.55) and (16.1,8.18) .. (18.0,8.18)
  --(21.72,8.18) .. controls (22.45,8.18) and (22.45,7.95) .. (22.45,7.35)
  --(22.45,6.75);
\draw[icPath,branchBrick] (15.015,12.25)
  .. controls (16.4,12.25) and (17.5,12.45) .. (20.03,12.45);
% B_5 and B_6 use J_2 and J_3.
\draw[icPath,branchYellow] (15.015,14.25)
  .. controls (16.1,14.25) and (16.6,14.35) .. (18.0,14.35)
  --(23.5,14.35) .. controls (24.15,14.35) and (24.15,14.1) .. (24.15,13.45)
  --(24.15,11.6) .. controls (24.15,11.05) and (23.8,11.05) .. (23.25,11.05)
  --(22.45,11.05);
\draw[icPath,branchYellow] (15.015,14.95)
  .. controls (16.8,14.95) and (17.2,16.75) .. (20.03,16.75);
\draw[icPath,branchPurple,icBridge] (15.015,15.9)
  .. controls (16.1,15.9) and (16.1,13.84) .. (18.0,13.84)
  --(21.72,13.84) .. controls (22.45,13.84) and (22.45,13.6) .. (22.45,13.05)
  --(22.45,12.45);
\draw[icPath,branchPurple] (15.015,16.6)
  .. controls (16.8,16.6) and (17.2,18.15) .. (20.03,18.15);
\foreach \yy in {5.85,6.55,7.2,7.9,9.9,10.6,11.55,12.25,14.25,14.95,15.9,16.6}{
  \node[icNeighbor] at (15.015,\yy) {};}

% The four shown local branch sets are pairwise adjacent in each model.
\foreach \upper/\lower in {5.35/6.75,11.05/12.45,16.75/18.15}{
  \draw[icEdge] (20.03,\upper)--(22.45,\upper);
  \draw[icEdge] (20.03,\lower)--(22.45,\lower);
  \draw[icEdge] (20.03,\upper)--(20.03,\lower);
  \draw[icEdge] (22.45,\upper)--(22.45,\lower);
  \draw[icEdge] (20.03,\upper)--(22.45,\lower);
  \draw[icEdge] (20.03,\lower)--(22.45,\upper);}
\foreach \xx/\yy/\number/\col in {
  20.03/5.35/1/branchBlue,20.03/6.75/2/branchRose,
  22.45/5.35/3/branchGreen,22.45/6.75/4/branchBrick,
  20.03/11.05/3/branchGreen,20.03/12.45/4/branchBrick,
  22.45/11.05/5/branchYellow,22.45/12.45/6/branchPurple,
  20.03/16.75/5/branchYellow,20.03/18.15/6/branchPurple,
  22.45/16.75/1/branchBlue,22.45/18.15/2/branchRose}{
  \node[icBlock,fill=\col!25] at (\xx,\yy) {\number};}
% Black boundary roots of the twelve local branch sets.
\foreach \xx/\yy in {
  19.3633/5.35,19.3633/6.75,
  23.1167/5.35,22.45/7.2083,
  19.3633/11.05,19.3633/12.45,
  23.1167/11.05,22.45/12.9083,
  19.3633/16.75,19.3633/18.15,
  23.1167/16.75,23.1167/18.15}{
  \node[icVertex] at (\xx,\yy) {};}
\foreach \yy in {7.3,13.0,18.7}{\node[icNote] at (21.24,\yy) {$\cdots$};}
% Label isolated portions of each branch, away from crossings.
\node[text=black] at (24.4,3.65) {$B_1$};
\node[text=black] at (24.65,18.75) {$B_2$};
\node[text=black] at (18.4,10.35) {$B_3$};
\node[text=black] at (18.4,12.90) {$B_4$};
\node[text=black] at (17.8,15.75) {$B_5$};
\node[text=black] at (17.8,18.35) {$B_6$};

% Keep the white-filled root vertices above their incident edges.
\foreach \yy in {6.2,7.55,10.25,11.9,14.6,16.25}{
  \node[icVertex] at (13.2,\yy) {};}
\end{tikzpicture}%
}
\endgroup
\caption{Illustration for the induction step. Ellipses labeled $i$ belong to the branch set $B_i$ of the $R$-rooted model.}
\label{fig:induction-step}
\end{figure}

We next construct the branch sets of the required $R$-rooted $K_p$ model, using $\mathcal R$ and three woven subgraphs of $H$.

By Inequalities~\eqref{eq:induction-connectivity}, $\delta(H)\ge Np\ge3p$.
Choose two neighbors of each vertex of $R_H$, all $2p$ choices distinct
and outside $R_H$: at each choice fewer than $3p$ vertices are forbidden.
Let $Y$ be the set of these neighbors.
By Inequality~\eqref{eq:induction-H-list} and the choice of
$C$ in Equation~\eqref{eq:choose-recursive},
$|L_H|-3p\ge|L|-(6N+3)p\ge42\cst{lem:separation}p$.
Color $R_H\cup Y$ with distinct colors from $L_H$, delete these vertices,
and remove from each remaining list the colors used on its neighbors
in $R_H\cup Y$.
The resulting bad sublist $L_0$ on $H\setminus(R_H\cup Y)$ has
$|L_0|\ge|L_H|-3p$.

Apply Lemma~\ref{lem:separation} with $t=14p$ twice, first to
$H\setminus(R_H\cup Y)$ and then to one of the resulting subgraphs.
The two bad sublists obtained in the first application have size at least
$28\cst{lem:separation}p$, so the second application is valid.
This produces three disjoint induced subgraphs $H_1,H_2,H_3$ with
bad sublists $L_i\subseteq L_0$ satisfying
$|L_i|\ge|L_0|-28\cst{lem:separation}p>4Np$.
By Lemma~\ref{thm:bad-connected}, each $H_i$ contains an induced
$Np$-connected subgraph $J_i$ with a bad sublist $L_i^*\subseteq L_i$.
By Equation~\eqref{eq:choose-recursive} and the assumption on $|L|$,
\[
 |L_i^*|\ge|L|-(10N+3+28\cst{lem:separation})p
 \ge C\left(3^a+\frac{2p}{3}\right).
\]
Since $\kappa(J_i)\ge Np\ge2Np/3$, the induction hypothesis for $b-1$
makes each $J_i$ $(2p/3,2p)$-woven.

By Inequalities~\eqref{eq:induction-connectivity} and the choice of $N$ in Equation~\eqref{eq:choose-recursive}, $\kappa(H\setminus R_H)\ge(N-1)p\ge10(2p)$. By Lemma~\ref{thm:linked}, $H\setminus R_H$ is $2p$-linked. The subgraphs $J_1,J_2,J_3$ are pairwise disjoint and lie in $H\setminus(R_H\cup Y)$. Since $b\ge1$, $p$ is divisible by $3$.

Partition $R_H$ into three equal sets $A_1,A_2,A_3$. For each $j\in[3]$, choose $U_j\subseteq V(J_j)$ of size $2p/3$ and assign its vertices bijectively to $R_H\setminus A_j$. For each $r_i^*\in R_H$, pair its two selected neighbors in $Y$ with its two assigned vertices in $U_1\cup U_2\cup U_3$. These $2p$ pairs have distinct ends. By the $2p$-linked property of $H\setminus R_H$, there is a linkage $\mathcal F$ in $H\setminus R_H$ joining these pairs.

By Lemma~\ref{lem:reroute}, applied successively to $J_1,J_2,J_3$ in $H\setminus R_H$ with root sets $U_1,U_2,U_3$, we may reroute $\mathcal F$ and choose a $K_{2p/3}$ model $\calM_j$ in each $J_j$ rooted at $U_j$ so that for each $j\in[3]$, $V(\mathcal F)\cap V(\calM_j)=U_j$. At each step, the new linkage uses only vertices of the preceding linkage and the current $J_j$. Since the $J_j$ are pairwise disjoint and the pairing remains fixed, later reroutings preserve this equality for earlier models.

For each $i\in[p]$, let $B_i$ be the union of the $(r_i,r_i^*)$-path in $\mathcal R$, the two selected edges from $r_i^*$ to $Y$, the two corresponding paths in $\mathcal F$, and the branch sets of the local models containing the two vertices assigned to $r_i^*$ (see Fig.~\ref{fig:induction-step} for an illustration). These subgraphs are connected and pairwise disjoint, since the local branch sets and paths assigned to different roots are disjoint and $V(\mathcal R)\cap V(H)=R_H$. Each $B_i$ contains exactly one vertex of $R$, namely $r_i$. Any two use a common local model, because each uses two of the three models $\calM_1,\calM_2,\calM_3$. Their local branch sets in that model are distinct and adjacent. Thus $B_1,\ldots,B_p$ form the required $R$-rooted $K_p$ model.

The paths in $\calQ$ avoid $H$ and are disjoint from $\mathcal R$. Hence $\calQ$ is a $\{(s_j,t_j)\}_{j\in[3p]}$-linkage disjoint from the model. By Claim~\ref{clm:disjoint-roots}, $G$ is $(p,3p)$-woven.

This completes the proof of Lemma~\ref{thm:recursive}.

\paragraph{Disclosure of AI use.}
We used AI tools to help combine
the result of OpenAI~\cite{OpenAI2026} with our earlier work~\cite{GuXu2026}
and to provide the proof of Lemma~\ref{thm:local-list}.
We have carefully checked all AI-generated material incorporated
into this paper and take full responsibility for its content.

\begingroup
\small
\sloppy

\endgroup
\end{document}